\documentclass[final, reqno]{amsart}

\title{Multiplicative Ehresmann connections for Lie groupoid fibrations}

\author{Matthijs Lau}
\thanks{Matthijs Lau: 
DipMat, Universit\`a degli Studi di Salerno,
via Giovanni Paolo II n$^\circ$ 123, 84084 Fisciano (SA), Italy, \texttt{mlau@unisa.it}}

\author{Ioan M\u{a}rcu\cb{t}}
\thanks{Ioan M\u{a}rcu\cb{t}: Mathematisches Institut, Universit\"{a}t zu K\"{o}ln,
Weyertal 86-90,
D-50931 K\"{o}ln, Germany}

\dedicatory{Dedicated to Rui Loja Fernandes on his 60th birthday}

\usepackage{graphicx}
\usepackage{comment}
\usepackage{combelow}
\usepackage{amssymb, amsmath}
\usepackage[alphabetic]{amsrefs}
\usepackage{mathrsfs}
\usepackage{enumerate}

\usepackage{graphicx, float}
\usepackage{tikz-cd}
\usepackage{tikz}
\usetikzlibrary{positioning, arrows, arrows.meta, bending, calc}

\usepackage{hyperref}

\newtheorem{theorem}{Theorem}[section]
\newtheorem{proposition}[theorem]{Proposition}
\newtheorem{corollary}[theorem]{Corollary}
\newtheorem{lemma}[theorem]{Lemma}
\newtheorem{question}[theorem]{Question}
\theoremstyle{definition}
\newtheorem{definitionx}[theorem]{Definition}

\theoremstyle{remark}
\newtheorem*{remarkx}{Remark}
\newtheorem{examplex}[theorem]{Example}
\newtheorem{examplesx}[theorem]{Examples}

\newcommand{\xqed}[1]{%
\leavevmode\unskip\penalty9999 \hbox{}\nobreak\hfill \quad\hbox{#1}}
\newcommand{\triangleqed}{\xqed{$\vartriangle$}}
\newcommand{\downtriangleqed}{\xqed{$\triangledown$}}

\newenvironment{definition}{\begin{definitionx}}{\downtriangleqed\end{definitionx}}

\newenvironment{remark}{\begin{remarkx}}{\triangleqed\end{remarkx}}
\newenvironment{example}{\begin{examplex}}{\triangleqed\end{examplex}}
\newenvironment{examples}{\begin{examplesx}}{\triangleqed\end{examplesx}}

\newcounter{char}
\loop\ifnum\value{char}<27
\expandafter\edef\csname gr\alph{char}\endcsname{\noexpand\boldsymbol{\alph{char}}} 
\expandafter\edef\csname gr\Alph{char}\endcsname{\noexpand\mathcal{\Alph{char}}} 
\expandafter\edef\csname bb\Alph{char}\endcsname{\noexpand\mathbb{\Alph{char}}} 
\expandafter\edef\csname scr\Alph{char}\endcsname{\noexpand\mathscr{\Alph{char}}} 
\expandafter\edef\csname scr\alph{char}\endcsname{\noexpand\mathscr{\alph{char}}} 
\addtocounter{char}{1}
\repeat

\def\mydeff#1{\expandafter\def\csname frk#1\endcsname{\mathfrak{#1}}}
\def\mydefallf#1{\ifx#1\mydefallf\else\mydeff#1\expandafter\mydefallf\fi}
\mydefallf abcdefghijklmnopqrstuvwxyz\mydefallf

\newcommand{\VB}{\ensuremath{\grV\grB}}
\newcommand{\LA}{\ensuremath{\grL\grA}}

\newcommand{\diffto}{\xrightarrow{\raisebox{-0.2 em}[0pt][0pt]{\smash{\ensuremath{\sim}}}}}

\newcommand{\tto}{\rightrightarrows}

\newcommand{\acts}[1][]{\ {\rotatebox[origin = c]{-90}{$\circlearrowright$}}_{#1}\ }

\newcommand{\fp}[2]{{}_{#1}\times_{#2}}

\newcommand{\sslash}{\mathbin{\mkern-2mu/\mkern-6mu/\mkern-2mu}}

\DeclareMathOperator{\Hor}{Hor}
\DeclareMathOperator{\hor}{hor}

\renewcommand{\tilde}{\widetilde}

\date{}
\usepackage{color}
\definecolor{tocolor}{rgb}{.1,.1,.5}
\definecolor{urlcolor}{rgb}{.2,.2,.6}
\definecolor{linkcolor}{rgb}{.1,.1,.6}
\definecolor{citecolor}{rgb}{.6,.2,.1}
\hypersetup{colorlinks = true, urlcolor = urlcolor, linkcolor = linkcolor, citecolor = citecolor}
\definecolor{darkgreen}{rgb}{0.0, 0.5, 0.0}

\begingroup
\hypersetup{linkcolor = blue}
\endgroup

\begin{document}
\maketitle
\begin{abstract}
	We introduce multiplicative Ehresmann connections on surjective submersions of Lie groupoids, extending both the classical notion of Ehresmann connections on fibre bundles and the more recent notion of multiplicative connections on Lie groupoid extensions. We investigate the existence of such connections, showing that, in general, they may fail to exist even for proper Lie groupoids. In contrast, positive results hold for Morita submersions, uniform Lie groupoid fibrations, locally trivial families of Lie groupoids, and proper families of Lie groupoids. Our main results concern completeness. For Lie groupoid fibrations, we prove that the completeness of a multiplicative connection is governed by the induced connection on the kernel bundle and, under connectivity assumptions, by the base connection. For families of source-proper Lie groupoids, we prove the equivalence between local triviality and the existence of complete multiplicative Ehresmann connections.
\end{abstract}

\setcounter{tocdepth}{1}
\tableofcontents

\section{Introduction}
In many areas of differential geometry, fibrations play a fundamental role---most notably vector bundles, principal bundles, covering spaces, Riemannian submersions, and symplectic fibrations. These classes of fibrations come with corresponding notions of connections, defined by certain compatibility conditions with the internal geometric or algebraic structure of the fibration. In \cite{Laurent-Gengoux2009}, the study of Lie groupoid extensions, i.e., surjective submersions of Lie groupoids covering the identity, and their corresponding notion of \emph{multiplicative} Ehresmann connection was initiated. The authors introduced these connections in order to describe non-abelian differentiable gerbes over stacks and studied, among other things, the problems of existence and completeness. With the aim of producing local models in Poisson geometry \cite{FerMa26}, the study of multiplicative Ehresmann connections was taken up again in \cite{Fernandes2023}, where a strong existence result was obtained for proper groupoids and where the infinitesimal theory was developed. The cohomological obstruction theory at both the groupoid and algebroid levels was further developed in \cite{Grad2025}. Multiplicative Ehresmann connections have found recent applications in the study of Yang–Mills theory on Lie groupoids \cite{GradPhD} and in the construction of a Cartan model for the differentiable stack cohomology of Lie groupoids \cite{KTPhD}.

The purpose of the present paper is to extend the theory of multiplicative Ehresmann connections beyond extensions to Lie groupoid morphisms
\[
	\pi\colon(\grG\tto M)\to (\grH\tto N)
\]
which are surjective submersions, but whose base map $\pi_0\colon{M}\to N$ is not necessarily the identity. Also in this setting, multiplicative Ehresmann connections are defined by requiring their horizontal bundle to be a \VB-subgroupoid of $T\grG\tto TM$
\[
	\Hor\tto \Hor_0,
\]
where $\Hor_0$ is an induced connection for the base map $\pi_0$. These objects simultaneously generalise classical Ehresmann connections on surjective submersions, linear connections on vector bundles, multiplicative connections on Lie groupoid extensions, and, in particular, principal connections, via the gauge groupoid.

Section \ref{sec:mEc} is devoted to the basic properties of multiplicative Ehresmann connections. We show that they are characterized by the property that the horizontal paths form a subgroupoid of the path groupoid of $\grG$ (Proposition \ref{prp: cond MC}). Finally, we briefly review the infinitesimal notions of multiplicative connections.

In Section \ref{sec:existence} and at the beginning of Section \ref{sec:loc_triv_fam}, we turn to the existence problem for multiplicative Ehresmann connections. For extensions, properness of the total groupoid suffices to ensure existence \cite{Fernandes2023}. The general case of Lie groupoid fibrations is much more restrictive. For example, the action morphism $\pi\colon H\ltimes M\to H$ corresponding to a non-trivial action of a connected group $H\acts M$ never admits a multiplicative connection (Corollary \ref{cor:action}). We obtain positive results for the following classes of groupoid morphisms (reviewed in Section \ref{ss:classes of maps}).

\begin{theorem}
	Multiplicative Ehresmann connections always exist for:
	\begin{enumerate}
		\item Morita fibrations (Proposition \ref{prp:Morita_fib}),

		\item proper uniform Lie groupoid fibrations (Corollary \ref{cor: proper uniform Lgrpd fib}),

		\item locally trivial families of Lie groupoids (Proposition \ref{prp: LC imply MC}), and

		\item proper families of Lie groupoids (Theorem \ref{thm:prop_fam}).
	\end{enumerate}
\end{theorem}

The main results of the paper are about the completeness of multiplicative Ehresmann connections. Theorem \ref{thm: compl ker imply compl total for Lie grpd fib}, which generalizes \cite{Fernandes2023} from extensions to fibrations, reduces the problem to the completeness of an induced connection on the kernel. Further, under the assumption of source-connectivity of the kernel, Corollary \ref{coro:base_of_complete_is_loc_triv} and Theorem \ref{thm: complete for s-conn} reduce the problem to completeness of the base connection. The following statement summarises the results from Section \ref{sec:compl}.

\begin{theorem}
	Any multiplicative Ehresmann connection $\Hor\tto \Hor_0$ on a Lie groupoid surjective submersion $\pi\colon \grG\to \grH$ induces a multiplicative Ehresmann connection $\Hor^{\grK}\tto \Hor_0$ on the kernel $\grK: = \ker\pi\to N$. The following implications hold regarding the completeness of these connections

	\begin{center}
		\begin{tikzpicture}[arrow/.style = {-implies, double, double distance = 2pt}, node distance = 1.2cm]
			\node (full) {$\Hor$ is complete};
			\node (kern) [right = of full] {$\Hor^{\grK}$ is complete};
			\node (base) [right = of kern] {$\Hor_0$ is complete};

			\draw[arrow] (full) to (kern);
			\draw[arrow] (kern) to (base);

			\draw[arrow, dashed] ($(kern.south)!0.5!(kern.south west)$) to[in = 345, out = 195] node[below, midway] {$\pi$ is a fibration} ($(full.south east)!0.5!(full.south)$);
			\draw[arrow, dashed] ($(base.south)!0.5!(base.south west)$) to[in = 345, out = 195] node[below, midway] {$\grK$ is source-connected} ($(kern.south east)!0.5!(kern.south)$);
		\end{tikzpicture}
	\end{center}
	where the dashed arrows assume the condition written underneath them.
\end{theorem}

Moreover, under the assumption of source-connectivity of $\grH$, we show that being a fibration is also a necessary condition for completeness (Proposition \ref{prp: complt MC and s-conn imply Lgrpd fib}).

The final Section \ref{sec:loc_triv_fam} discusses a groupoid analogue of the equivalence between locally trivial fibrations and surjective submersions admitting complete Ehresmann connections---a result, whose correct proof was found only recently by Mathias del Hoyo \cite{delHoyo2016}. By using the same techniques, we obtain a result for source-proper families of Lie groupoids in Theorem \ref{thm: s-proper implies existence of multiplicative connection}, which we state as an equivalence.
\begin{theorem}
	Let $\pi\colon(\grG\tto M)\to N$ be a family of Lie groupoids, where $\grG$ is a source-proper Lie groupoid. Then $\pi$ is a locally trivial family of Lie groupoids if and only if it admits a complete multiplicative Ehresmann connection.
\end{theorem}

Appendix~\ref{app: Lgrpd fib} contains a discussion on the relation between notions of Lie groupoid fibration; in particular, Theorem \ref{thm:weaking_fibration} gives a set of easily checkable conditions for a Lie groupoid surjective submersion to be a fibration. Appendix~\ref{app: proofs} is devoted to Proposition~\ref{prp: cond VBsubgrpd}, which gives a criterion for identifying \VB-subgroupoids.

\subsection*{Acknowledgements}
We would like to thank Rui Loja Fernandes, Matias del Hoyo, Madeleine Jotz, and \v{Z}an Grad for useful discussions about multiplicative Ehresmann connections, and for their general interest in this project. We are also very grateful to Jo\~{a}o Mestre, who helped us understand the techniques of \cite{CrMeStr20}, which were used in the proof of Theorem \ref{thm:prop_fam}, and to Luca Vitagliano, who found for us Example \ref{ex:Luca}.

\section{Preliminaries}
\subsection{Assumptions and conventions}
Throughout the paper, all Lie groupoids are assumed to be Hausdorff, ensuring the uniqueness of solutions to the relevant flow problems. We denote a Lie groupoid by $\grG\tto M$, where $\grG$ is the space of arrows, $M$ is the space of objects. The structure maps are written as $\grs$ = source, $\grt$ = target, $\grm$ = multiplication, $\gru$ = unit, and $\gri$ = inversion. The source-fibres will be simply denoted by $\grs^{-1}(x)$, and we will avoid notations like $\grG_{x}$ or $\grG(x,-)$; similarly for the target-fibres. We will identify the base with the unit section $M\cong \gru(M)\subset \grG$.

\subsection{Classes of groupoid maps}
\label{ss:classes of maps} We introduce the main classes of Lie groupoid morphisms used in this paper. The terminology follows \cite{Mackenzie2005}*{Section 2.6} and \cite{delHoyo2019}*{Section 2.1}.
\begin{definition}
	\label{def:grp_fibr} A Lie groupoid morphism $\pi\colon \grG \to \grH$ with base map $\pi_0\colon M\to N$ such that $\pi$ (and therefore also $\pi_0$) is a surjective submersion is called
	\begin{enumerate}
		\item a \emph{Lie groupoid surjective submersion};

		\item a \emph{Lie groupoid fibration}, if additionally the map
		      \begin{equation}
			      \label{eq: fibration}\pi\times \grs \colon \grG\to \grH \fp{\grs}{\pi_0}M,\qquad g\mapsto (\pi(g),\grs(g))
		      \end{equation}
		      is a surjective submersion;

		\item an \emph{action morphism}, if the map from \eqref{eq: fibration} is a diffeomorphism;

		\item a \emph{family of Lie groupoids}, if $\mathcal{H} = N\tto N$ is the unit groupoid;

		\item a \emph{uniform Lie groupoid fibration}, if the map
		      \begin{equation}
			      \label{eq: unifibration}\pi^{*}\colon \grG\to \pi_0^{*}\grH,\qquad g \mapsto (\grt(g),\pi(g),\grs(g))
		      \end{equation}
		      is a surjective submersion, where $\pi_0^{*}\grH$ is the pullback groupoid:
		      \[
			      \pi_0^{*}\grH: = M\fp{\pi_0}{\grt}\grH\fp{\grs}{\pi_0}M\tto M;
		      \]

		\item a \emph{Morita fibration} (or \emph{inductor}), if the map in \eqref{eq: unifibration} is an isomorphism;

		\item a \emph{Lie groupoid extension}, if the base map is the identity map, i.e.\ $\pi_0 = \mathrm{id}_{M}$.
	\end{enumerate}
\end{definition}

In Appendix \ref{app: Lgrpd fib} we discuss the connections between these notions. In particular, we have the following strict implications:
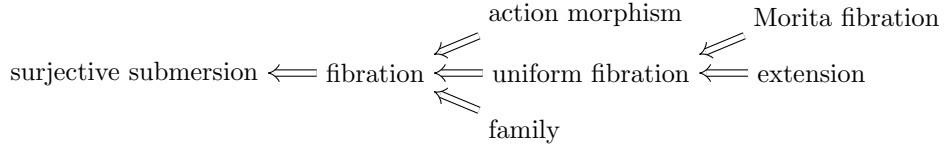
\begin{center}
	\begin{tikzpicture}[arrow/.style = {-implies, double, double distance = 2pt}, node distance = .65cm]
		\node (sub) {surjective submersion};
		\node (fib) [right = of sub] {fibration};
		\node (uni) [right = of fib] {uniform fibration};
		\node (ext) [right = of uni] {extension};

		\node (act) [above right = .25cm and .6cm of fib] {action morphism};
		\node (fam) [below right = .25cm and .6cm of fib] {family};
		\node (mor) [above right = .25cm and .6cm of uni] {Morita fibration};

		\draw[arrow] (fib) -- (sub);
		\draw[arrow] (uni) -- (fib);
		\draw[arrow] (ext) -- (uni);

		\draw[arrow] (act.south west) -- (fib.north east);
		\draw[arrow] (fam.north west) -- (fib.south east);
		\draw[arrow] (mor.south west) -- (uni.north east);
	\end{tikzpicture}
\end{center}
The kernel $\ker \pi: = \pi^{-1}(\gru(N))\tto M$ of a Lie groupoid surjective submersion $\pi\colon \grG \to \grH$ is a subgroupoid of $\grG$. Furthermore, the restriction $\underline{\pi}\colon\ker\pi\to N$ endows $\ker\pi$ with the structure of a family of Lie groupoids parametrised by $N$, the base of $\grH$. Note that, beyond the case of extensions, the information of the morphism is not contained exclusively in the kernel. This is discussed in detail in \cite{Mackenzie2005}, from where we borrow the following example.
\begin{example}
	Consider a principal bundle $G\acts M\to N$. The groupoid map
	\[
		\pi\colon M\times M\to (M\times M)/G,
	\]
	from the pair groupoid to the gauge groupoid is a uniform Lie groupoid fibration whose kernel is just the identity subgroupoid.
\end{example}

Compared to general Lie groupoid surjective submersions, the situation for Lie groupoid fibrations is considerably better, as the data of the morphism is encoded by a so-called normal subgroupoid system, see again \cite{Mackenzie2005}*{Section~2.4}. This is the deeper reason why we can reduce the completeness problem for multiplicative connections to the kernel (see Theorem \ref{thm: compl ker imply compl total for Lie grpd fib}).

\subsection{\VB-groupoids}
In order to define multiplicative connections, we quickly recall the basics of \VB-groupoids (for a general reference, see \cite{Mackenzie2005} and \cite{BuCadH2016}).

\begin{definition}
	A \emph{\VB-groupoid} is a tuple $(\Gamma, E \sslash \grG, M)$ such that:
	\begin{itemize}
		\item $\Gamma \tto E$ and $\grG \tto M$ are Lie groupoids;

		\item $\tilde{q}\colon \Gamma \to \grG$ and $q \colon E \to M$ are vector bundles;

		\item $\tilde{q}$ is a Lie–groupoid morphism over $q$, and the addition map
		      \[
			      + \colon \Gamma\fp{\tilde{q}}{\tilde{q}}\Gamma \longrightarrow \Gamma
		      \]
		      is a Lie–groupoid morphism covering the addition map
		      \[
			      + \colon E\fp{q}{q}E \longrightarrow E.
		      \]
	\end{itemize}
\end{definition}
A \VB-groupoid $(\Gamma, E \sslash \grG, M)$ will be represented by the diagram:
\[
	\begin{tikzcd}
		\Gamma & E \\
		\grG & M \arrow[shift left, from = 1-1, to = 1-2]
		\arrow[shift right, from = 1-1, to = 1-2]
		\arrow[from = 1-1, to = 2-1]
		\arrow[from = 1-2, to = 2-2]
		\arrow[shift left, from = 2-1, to = 2-2]
		\arrow[shift right, from = 2-1, to = 2-2]
	\end{tikzcd}
\]
We call the Lie groupoid $\Gamma \tto E$ the \emph{total groupoid}, whose structure maps we decorate with a tilde, and the projection $\tilde{q}\colon \Gamma \to \grG$ is the \emph{total vector bundle}. The groupoid $\grG \tto M$ is the \emph{base groupoid}, and $E \to M$ is the \emph{side bundle}. When no ambiguity may occur, we refer to a \VB-groupoid simply by its total space $\Gamma$; when it is useful to emphasise the base, we say that $\Gamma$ is a \emph{\VB-groupoid over $\grG \tto M$}.

A \emph{\VB-groupoid morphism} from $(\Gamma, E \sslash \grG, M)$ to $(\Omega, F \sslash \grH, N)$ is a smooth map $\Phi \colon \Gamma \to \Omega$ which is simultaneously a vector bundle morphism over a Lie groupoid morphism $\phi\colon \grG \to \grH$ and a Lie groupoid morphism $(\Gamma \tto E) \to (\Omega \tto F)$.

A \emph{\VB-subgroupoid} of $\Gamma$ is a \VB-groupoid $\Omega$ together with a \VB-groupoid morphism $\iota\colon \Omega \to \Gamma$ whose underlying map of manifolds is an injective immersion and which identifies $\Omega$ with a vector subbundle of $\Gamma$. Alternatively, these can simply be described as a simultaneous vector subbundle and subgroupoid, see Proposition~\ref{prp: cond VBsubgrpd}.

Consider a \emph{short exact sequence} of \VB-groupoids, all over the same groupoid $\grG\tto M$ and covering the identity morphisms:
\[
	\begin{tikzcd}
		0 & \Gamma & {\Gamma'} & {\Gamma''} & 0 \\
		0 & E & {E'} & {E''} & 0
		\arrow[from = 1-1, to = 1-2]
		\arrow["\iota", hook, from = 1-2, to = 1-3]
		\arrow[shift left, from = 1-2, to = 2-2]
		\arrow[shift right, from = 1-2, to = 2-2]
		\arrow["\pi", two heads, from = 1-3, to = 1-4]
		\arrow[shift left, from = 1-3, to = 2-3]
		\arrow[shift right, from = 1-3, to = 2-3]
		\arrow[from = 1-4, to = 1-5]
		\arrow[shift left, from = 1-4, to = 2-4]
		\arrow[shift right, from = 1-4, to = 2-4]
		\arrow[from = 2-1, to = 2-2]
		\arrow["{{\iota_0}}", hook, from = 2-2, to = 2-3]
		\arrow["{{\pi_0}}", two heads, from = 2-3, to = 2-4]
		\arrow[from = 2-4, to = 2-5]
	\end{tikzcd}
\]

The exactness of the bottom row is implied by that of the top row, as the source maps are surjective and they admit the unit sections as a right inverse. Therefore, we typically only write the top row of a short exact sequence of \VB-groupoids.

Note also that the image of the injective \VB-map $\iota$ is a \VB-subgroupoid of $\Gamma'$, and that $\Gamma''$ is determined up to isomorphism as $\Gamma'/\iota(\Gamma)$. Similarly, if we are only given the surjective \VB-map $\pi$ over $\mathrm{id}_{\grG}$, then it follows that $\ker \pi$ is a \VB-subgroupoid (see \cite{LiBland2014}*{Corollary~C.3} or Proposition \ref{prp: cond VBsubgrpd}).

As for vector bundles (e.g.~\cite{Tu2017}*{Theorem 27.20}), \emph{split} exact sequences of \VB-groupoids are characterised as follows.
\begin{lemma}
	\label{lem: splitting} Consider a short exact sequence of \VB-groupoids over $\grG\tto M$, say
	\[
		\begin{tikzcd}
			0 & \Gamma & {\Gamma'} & {\Gamma''} & 0
			\arrow[from = 1-1, to = 1-2]
			\arrow["\iota", hook, from = 1-2, to = 1-3]
			\arrow["\pi", two heads, from = 1-3, to = 1-4]
			\arrow[from = 1-4, to = 1-5]
		\end{tikzcd}
	\]
	Then, there is a $1$-$1$ correspondence between the following objects:
	\begin{enumerate}
		\item \label{item: splitting; horizontal lift} \VB-groupoid morphisms $h\colon \Gamma''\to\Gamma'$ such that $\pi\circ h = \mathrm{id}_{\Gamma''}$.

		\item \label{item: splitting; vertical projection} \VB-groupoid morphisms $p\colon\Gamma'\to\Gamma$ such that $p\circ \iota = \mathrm{id}_{\Gamma}$.

		\item \label{item: splitting; complement} \VB-subgroupoids $C\subset \Gamma'$ complementing $\iota(\Gamma)$ as a \VB-groupoid.

		\item \label{item: splitting; splitting} Splittings, i.e.\ \VB-groupoid morphisms $\Phi\colon \Gamma'\to\Gamma\oplus\Gamma''$ fitting into the commutative diagram
		      \[
			      \begin{tikzcd}
				      0 & \Gamma & {\Gamma'} & {\Gamma''} & 0 \\
				      0 & \Gamma & {\Gamma\oplus\Gamma''} & {\Gamma''} & 0
				      \arrow[from = 1-1, to = 1-2]
				      \arrow[from = 1-4, to = 1-5]
				      \arrow["\iota", hook, from = 1-2, to = 1-3]
				      \arrow["\pi", two heads, from = 1-3, to = 1-4]
				      \arrow[from = 2-1, to = 2-2]
				      \arrow[from = 2-4, to = 2-5]
				      \arrow["{\mathrm{incl}_1}", hook, from = 2-2, to = 2-3]
				      \arrow["{\mathrm{pr}_2}", two heads, from = 2-3, to = 2-4]
				      \arrow["\Phi", from = 1-3, to = 2-3]
				      \arrow[from = 1-2, to = 2-2, shift left = .5, no head]
				      \arrow[from = 1-2, to = 2-2, shift right = .5, no head]
				      \arrow[from = 1-4, to = 2-4, shift left = .5, no head]
				      \arrow[from = 1-4, to = 2-4, shift right = .5, no head]
			      \end{tikzcd}
		      \]
		      (Remark: it follows that $\Phi$ is an isomorphism.)
	\end{enumerate}
	The correspondences are determined uniquely by
	\[
		h\circ\pi + \iota\circ p = \mathrm{id}_{\Gamma'},\qquad C = \ker p = \mathrm{im} h,\qquad \Phi = (p,\pi),\qquad \Phi^{-1} =  \iota\oplus h.
	\]
	Moreover, in (3), $\pi|_{C}\colon C\to \Gamma''$ is an isomorphism of \VB-groupoids.
\end{lemma}

\begin{proof}
	The analogous statement for vector bundles is standard, and so it implies that the correspondences are uniquely determined. We need only show that these correspondences result in multiplicative maps and \VB-groupoids.

	\eqref{item: splitting; horizontal lift}$\iff$\eqref{item: splitting; splitting}: Suppose $h\colon \Gamma ''\to\Gamma'$ is a right splitting of the short exact sequence, such that $\pi\circ h = \mathrm{id}_{\Gamma''}$. We define the isomorphism of vector bundles
	\[
		\iota\oplus h\colon \Gamma\oplus\Gamma''\to \Gamma'\colon (\gamma,\gamma'')\mapsto \iota(\gamma) + h(\gamma'').
	\]
	As $h$, $\iota$, and $+$ are \VB-groupoid maps, this will define a \VB-groupoid map as well. Defining $\Phi = (\iota\oplus h)^{-1}$ gives a map satisfying condition~\eqref{item: splitting; splitting}.

	Conversely, suppose that $\Phi\colon\Gamma'\to\Gamma\oplus\Gamma''$ is a splitting of the short exact sequence. We can then define a right inverse to $\pi$ as $h\colon\Gamma''\to\Gamma'\colon \gamma''\mapsto \phi^{-1}(\mathrm{incl}_{2}(\gamma''))$.

	\eqref{item: splitting; vertical projection}$\iff$\eqref{item: splitting; splitting}: Suppose that $p\colon \Gamma'\to\Gamma$ is a left splitting, such that $p\circ\iota = \mathrm{id}_{\Gamma}$. We define $\Phi$ as the map $(p,\pi)\colon \Gamma'\to\Gamma\oplus\Gamma''\colon \gamma'\mapsto (p(\gamma'),\pi(\gamma'))$, which is a \VB-groupoid map by the universal property of the direct sum. One easily verifies that this satisfies condition~\eqref{item: splitting; splitting}.

	Conversely, given a splitting $\Phi\colon\Gamma'\to \Gamma\oplus\Gamma''$ we can define $p\colon\Gamma'\to\Gamma = \mathrm{pr}_{1}\circ\Phi$.

	\eqref{item: splitting; complement}$\iff$\eqref{item: splitting; splitting}: One readily verifies this by chasing the diagram.
\end{proof}

\subsection{Multiplicative distributions}

The prototypical example of a \VB-groupoid is the tangent bundle of a Lie groupoid $\grG\tto M$:
\[
	\begin{tikzcd}
		{T\grG} & TM \\
		\grG & M
		\arrow[shift left, from = 1-1, to = 1-2]
		\arrow[shift right, from = 1-1, to = 1-2]
		\arrow[from = 1-1, to = 2-1]
		\arrow[from = 1-2, to = 2-2]
		\arrow[shift left, from = 2-1, to = 2-2]
		\arrow[shift right, from = 2-1, to = 2-2]
	\end{tikzcd}
\]

Distributions and foliations in the multiplicative setting have been studied by several authors, e.g., \ in \cites{Crainic_Salazar, Fernandes2023, Hawkins2008, Jotz2012, JoOr2014, Laurent-Gengoux2009, Xiang2006}.

\begin{definition}
	\label{dfn: mult fol} A \emph{multiplicative distribution} on a Lie groupoid $\grG\tto M$ is a distribution $\mathcal{D}$ on $\grG$ that is also a \VB-subgroupoid $\mathcal{D}\tto \mathcal{D}_0$ of $T\grG\tto TM$.

	If $\mathcal{D}$ is involutive, we will call it a \emph{multiplicative foliation}.
\end{definition}

\begin{remark}
	Multiplicative foliation $\mathcal{D}$ on $\grG\tto M$, the base $\mathcal{D}_0\subset TM$ is also an involutive distribution \cite{Hawkins2008}*{Lemma 5.1}; hence a foliation on $M$. In fact, $(\mathcal{D}, \mathcal{D}_0\sslash \grG, M)$ forms a so-called \LA-groupoid, i.e.\ $\mathcal{D}$ and $\mathcal{D}_0$ are Lie algebroids and all groupoid structure maps are Lie algebroid maps (see \cite{Mackenzie1992}). Multiplicative foliations therefore coincide with \LA-subgroupoid of $T\grG\tto TM$.
\end{remark}

Note that a multiplicative distribution gives rise to a short exact sequence of \VB-groupoids, namely:
\[
	\begin{tikzcd}
		0 & {\mathcal{D}} & {T\grG} & {T\grG/\mathcal{D}} & 0 \\
		0 & {\mathcal{D}_0} & TM & {TM/\mathcal{D}_0} & 0
		\arrow[from = 1-1, to = 1-2]
		\arrow[hook, from = 1-2, to = 1-3]
		\arrow[shift left, from = 1-2, to = 2-2]
		\arrow[shift right, from = 1-2, to = 2-2]
		\arrow["q", two heads, from = 1-3, to = 1-4]
		\arrow[shift left, from = 1-3, to = 2-3]
		\arrow[shift right, from = 1-3, to = 2-3]
		\arrow[from = 1-4, to = 1-5]
		\arrow[shift left, from = 1-4, to = 2-4]
		\arrow[shift right, from = 1-4, to = 2-4]
		\arrow[from = 2-1, to = 2-2]
		\arrow[hook, from = 2-2, to = 2-3]
		\arrow["{{q_0}}", two heads, from = 2-3, to = 2-4]
		\arrow[from = 2-4, to = 2-5]
	\end{tikzcd}
\]

\section{Multiplicative Ehresmann connections}
\label{sec:mEc}

We introduce the main notion studied in this paper. Consider a Lie groupoid surjective submersion $\pi\colon \grG\to\grH$ with base map $\pi_0\colon M\to N$. Notice that $\ker T\pi$ defines a multiplicative distribution on $\grG$; moreover, it gives rise to a short exact sequence of \VB-groupoids:
\[
	\begin{tikzcd}[column sep = large]
		0 & \ker T\pi & {T\grG} & {T\grH\fp{}{\grH}\grG} & 0 \\
		0 & \ker T\pi_0 & TM & {TN\fp{}{N}M} & 0
		\arrow[from = 1-1, to = 1-2]
		\arrow["\iota", hook, from = 1-2, to = 1-3]
		\arrow[shift right = 1, from = 1-2, to = 2-2]
		\arrow[shift left = 1, from = 1-2, to = 2-2]
		\arrow["T\pi\times \mathrm{pr}", two heads, from = 1-3, to = 1-4]
		\arrow[shift right = 1, from = 1-3, to = 2-3]
		\arrow[shift left = 1, from = 1-3, to = 2-3]
		\arrow[from = 1-4, to = 1-5]
		\arrow[shift right = 1, from = 1-4, to = 2-4]
		\arrow[shift left = 1, from = 1-4, to = 2-4]
		\arrow[from = 2-1, to = 2-2]
		\arrow["\iota_0", hook, from = 2-2, to = 2-3]
		\arrow["T\pi_0\times\mathrm{pr}", two heads, from = 2-3, to = 2-4]
		\arrow[from = 2-4, to = 2-5]
	\end{tikzcd}
\]
where $T\grH\fp{}{\grH}\grG\tto TN\fp{}{N}M$ is a fibre product of Lie groupoids.

\begin{definition}
	A \emph{multiplicative Ehresmann connection (MEC)} is defined in either of the following ways, which are equivalent according to Lemma \ref{lem: splitting}:
	\begin{itemize}
		\item \emph{Multiplicative horizontal lift}: A \VB-groupoid morphism
		      \begin{align*}
			      \hor\colon T\grH\fp{}{\grH}\grG\to T\grG \quad                                                                                 & \textrm{such that}\quad (T\pi\times \mathrm{pr})\circ \hor = \mathrm{id}_{T\grH\fp{}{\grH}\grG}. \\
			      \intertext{\item \emph{Multiplicative vertical projection}: A \VB-groupoid morphism}\mathrm{vert}\colon T\grG\to\ker T\pi\quad & \textrm{such that}\quad \mathrm{vert}\circ \iota = \mathrm{id}_{\ker T\pi}.                      \\
			      \intertext{\item \emph{Multiplicative Ehresmann connection}: A \VB-subgroupoid}\Hor\subset T\grG \quad                         & \textrm{such that}\quad T\grG =  \Hor\oplus\ker T\pi.
		      \end{align*}
	\end{itemize}
	Then $\Hor_0\subset TM$, the base of the \VB-groupoid $\Hor\tto \Hor_0$, defines a usual Ehresmann connection for the base map $\pi_0\colon M\to N$. The base of the morphism $\hor$ is the horizontal lift $\hor_0\colon TN\fp{}{N}M\to TM$ with respect to $\Hor_0$.
\end{definition}

\begin{examples}\label{ex: MC}
	The following concepts fit into this framework.
	\begin{enumerate}
		\item If one considers a surjective submersion of manifolds $q\colon M\to B$ as a surjective submersion of Lie groupoids between the unit groupoids, then MECs are just usual Ehresmann connections.

		\item For a vector bundle $\pi\colon E\to M$, viewed as a groupoid map from the bundle of abelian groups to the unit groupoid, MECs are the same as linear connections.

		\item For Lie groupoid extensions $\pi\colon\grG\to\grH$, i.e.\ $\pi_0 =  \mathrm{id}_{M}$, MECs were studied in detail in \cite{Laurent-Gengoux2009}, and later in \cite{Fernandes2023}. In particular, they extend connection forms on principal bundles via the gauge groupoid construction.

		\item Multiplicative, affine connections on a Lie groupoid $\grH\tto N$ have been introduced and studied in \cite{Pugliese2023}. These are the same as affine connections $\nabla$ on $\grH$ which, when seen as an Ehresmann connection on the surjective Lie groupoid surjective submersion $\pi\colon T\grH\to \grH$, are multiplicative in our sense, see \cite{Pugliese2023}*{Proposition~3.3}. The authors of \cite{Pugliese2023} also developed cohomological obstruction theory for their existence, and show that these objects are very restrictive. In particular, their existence implies that $\grH$ is regular and that the connected components of its isotropy groups are abelian.
	\end{enumerate}
\end{examples}

For later use, let us remark that multiplicative Ehresmann connections lift infinitesimal groupoid automorphisms to infinitesimal groupoid automorphisms (for the basics on multiplicative vector fields on Lie groupoids, see \cite{MackXu98}).
\begin{proposition}\label{prop:lift:multipl}
	A multiplicative horizontal lift $\hor\colon T\grH\fp{}{\grH}\grG\to T\grG$ on $\pi\colon\grG\to\grH$ sends multiplicative vector fields to multiplicative vector fields
	\[
		\hor\colon \mathfrak{X}_{\mathrm{mult}}(\grH)\to\mathfrak{X}_{\mathrm{mult}}(\grG).
	\]
\end{proposition}
\begin{proof}
	A vector field $X$ on $\grH$ is multiplicative precisely when $X\colon \grH\to T\grH$ is a groupoid morphism. If this is the case, then so is the pullback section, i.e.\ the map
	\[
		\tilde{X}\colon \grG\to T\grH\times_{\grH}\grG\colon g\mapsto (X_{\pi(g)},g).
	\]
	The vector field $\hor (X)$ is obtained by composing this map with the \VB-groupoid morphism $\hor\colon T\grH\times_{\grH}\grG\to T\grG$. Hence, the result.
\end{proof}
The above proposition does not have a converse and therefore does not fully characterise a multiplicative Ehresmann connection. The following example was constructed by Luca Vitagliano. It shows that, in general, the conclusion of Proposition \ref{prop:lift:multipl} is not sufficient to ensure multiplicativity.
\begin{example}\label{ex:Luca}
	Consider the surjective submersion $\pi\colon\bbR^2\to S^1,\,(x,y)\mapsto e^{ix}$, and remark that this is a Lie group(oid) homomorphism with respect to the standard structures. Any multiplicative vector field on $S^1$ must be the zero vector field. Therefore, any connection $\hor$ on $\pi$ will trivially induce a map
	\[
		\hor\colon \mathfrak{X}_{\mathrm{mult}}(S^1)\to \mathfrak{X}_{\mathrm{mult}}(\bbR^2).
	\]s
	As $\bbR^2$ is an abelian Lie group, the left and right trivialization of $T\bbR^2$ coincide. Moreover, as a Lie group, it is isomorphic to the standard structure on $\bbR^{4}$ via the left/right trivialization. In this global chart, the vertical bundle is given by:
	\[
		\ker T\pi = \mathrm{span}\{(x,y,a,b)\in\bbR^{4}\mid a = 0\}
	\]
	Next, consider a distribution, given in the global chart as:
	\[
		\Hor = \{(x,y,a,b)\in\bbR^{4}\mid b = x^2a\}.
	\]
	Fibrewise, meaning for fixed $(x,y)\in\bbR^2$, it defines a linear complement and thus a connection. However, it is easy to see that $\Hor$ is not a subgroup of $T\bbR^2$ as for any $(x,y,a,b)\in\Hor$, we have that $2(x,y,a,b) = (2x,2y,2a,2b)\notin\Hor$.
\end{example}

\subsection{Lifting of paths}
In this subsection, we interpret multiplicativity of a connection from a geometric perspective, using the lifting of paths. In particular, we obtain characterizations similar to \cite{Xiang2006}*{Definition 2.1}, and generalizing \cite{Laurent-Gengoux2009}*{Proposition 4.3}.

Fix a Lie groupoid surjective submersion $\pi\colon \grG\to \grH$ with an Ehresmann connection
\[
	T\grG = \Hor\oplus \ker T\pi,
\]
and denote the horizontal lift operator by $\hor\colon T\grH\fp{}{\grH}\grG \to T\grG$. Given a curve $\gamma\colon I\to \grH$, where we will always denote $I = [0,1]$, with $h: =  \gamma(0)$ and a point $g\in\pi^{-1}(h)$, there exists a unique \emph{horizontal lift of $\gamma$ through $g$}, namely a curve
\[
	\tau_{\gamma}(g)\colon I^{0}\to M,\qquad t\mapsto \tau^{t}_{\gamma}(g)
\]
solving the ordinary differential equation
\[
	\begin{cases}
		\tfrac{d}{dt}\tau_{\gamma}^{t}(g) = \hor(\tau_{\gamma}^{t}(g),\tfrac{d}{dt}\gamma(t)), & t\in I, \\
		\tau_{\gamma}^{0}(g) = g,
	\end{cases}
\]
where $0\in I^{0}\subset I$ is the maximal domain of the solution. The connection is \emph{complete} precisely when $I^{0} = I$, for all $\gamma$ and $g$ as above.

The \emph{holonomy} of $\gamma$ or \emph{parallel transport} along $\gamma$ is the map
\[
	\tau_{\gamma}^{t}\colon U^{t}\subset \pi^{-1}(\gamma(0))\to \pi^{-1}(\gamma(t)), \qquad g\mapsto \tau_{\gamma}^{t}(g),
\]
where $U^{t}$ is the maximal domain on which it is defined.

Additionally, in the following description of multiplicativity in terms of holonomy, we use the notion of the so-called \emph{current groupoid} structure on the space of smooth paths in a Lie groupoid $\grG\tto M$, studied in \cite{Habib2020}:
\[
	C^{\infty}(I,\grG)\tto C^{\infty}(I,M),
\]
where the structure maps are pointwise, i.e.
\[
	\grs(\gamma)(t) = \grs(\gamma(t)), \ \ \grm(\gamma,\eta)(t) = \grm(\gamma(t),\eta(t)),\ \ 1_{\gamma}(t) = 1_{\gamma(t)},\ \ \gamma^{-1}(t) = (\gamma(t))^{-1}.
\]

\begin{proposition}\label{prp: cond MC}
	Let $\pi\colon\grG\to\grH$ be a Lie groupoid surjective submersion covering $\pi_0\colon M\to N$, and suppose $\Hor$ and $\Hor_0$ be connections for $\pi$ and $\pi_0$, respectively. The following are equivalent:
	\begin{enumerate}
		\item \label{item: cond MC; mult} $\Hor\tto \Hor_0$ is a multiplicative Ehresmann connection.

		\item \label{item: cond MC; lift} The following hold:
		      \begin{enumerate}
			      \item\label{item: cond MC; lift; u} For any $\delta\in C^{\infty}(I,N)$ and any $x\in M$ with $\pi_0(x) = \delta(0)$, if the $\Hor_0$-lift $\tau^{t}_{\delta}(x)$ is defined, then
			            \begin{align*}
				            \gru(\tau_{\delta}^{t}(x))         & = \tau^{t}_{\gru(\delta)}(\gru(x)); \intertext{\item for any $\gamma\in C^{\infty}(I,\grH)$ and any $g\in \grG$ with $\pi(g) = \gamma(0)$, if the $\Hor$-lift $\tau^{t}_{\gamma}(g)$ is defined, then}\grs(\tau^{t}_{\gamma}(g)) = \tau^{t}_{\grs(\gamma)}(\grs(g))                                                                            & ,\quad \grt(\tau^{t}_{\gamma}(g)) = \tau^{t}_{\grt(\gamma)}(\grt(g)), \\
				            \gri\big(\tau_{\gamma}^{t}(g)\big) & = \tau_{\gri(\gamma)}^{t}(\gri(g)); \intertext{\item\label{item: cond MC; lift; m} and for any $\eta\in C^{\infty}(I,\grH)$ with $\grt(\eta) = \grs(\gamma)$, and any $k\in \grG$ with $\grt(k) = \grs(g)$ and $\pi(k) = \eta(0)$, if the $\Hor$-lift $\tau^{t}_{\eta}(h)$ is also defined, then}\grm(\tau_{\gamma}^{t}(g),\tau_{\eta}^{t}(k)) & = \tau_{\grm(\gamma,\eta)}^{t}(\grm(g,k)).
			            \end{align*}
		      \end{enumerate}
	\end{enumerate}
\end{proposition}
\begin{proof}

	\eqref{item: cond MC; mult}$\implies$\eqref{item: cond MC; lift}: Suppose $\Hor\tto\Hor_0$ is a MEC.

	(a): Let $\delta\colon I\to N$ and $x\in M$ be as in the statement. Since  $T\gru\colon TM\to T\grG$ is the unit section of $T\grG$, we have that $T\gru (\Hor_0) = TM\cap \Hor$. This yields
	\[
		\tfrac{d}{dt}\gru(\tau^{t}_{\delta}(x)) =  T\gru \big(\tfrac{d}{dt}\tau^{t}_{\delta}(x)\big)\in T_{\tau^{t}_{\delta}(x)}M\cap \Hor_{\gru(\tau^t_{\delta}(x))}.
	\]
	So the path $t\mapsto \gru\circ \tau^{t}_{\delta}(x)$ is $\Hor$-horizontal. It starts at $\gru(x)$ and satisfies
	\[
		\pi\circ\gru\circ \tau^{t}_{\delta}(x) =  \gru\circ\pi_0\circ \tau^{t}_{\delta}(x) = \gru(\delta)(t).
	\]
	This implies the claimed equality.

	(b): Let $\gamma\colon I\to\grH$ and $g\in\grG$ be as in the statement. The path $t\mapsto \grs(\tau_{\gamma}^{t}(g))\in M$ satisfies:
	\[
		\tfrac{d}{dt}\grs(\tau_{\gamma}^{t}(g)) = T\grs\big(\tfrac{d}{dt}\tau_{\gamma}^{t}(g)\big)\in T\grs(\Hor_{\tau_{\gamma}^t(g)}) = ({\Hor_0})_{\grs(\tau_{\gamma}^t(g))},
	\]
	so it is $\Hor_0$-horizontal. Its base path is
	\[
		\pi_0\circ\grs(\tau_{\gamma}^{t}(g)) = \grs(\pi\circ \tau_{\gamma}^{t}(g)) = \grs(\gamma)(t),
	\]
	which implies that it is the $\Hor_0$-horizontal lift of $\grs\circ \gamma$:
	\[
		\grs(\tau_{\gamma}^{t}(g)) = \tau_{\grs(\gamma)}^{t}(\grs(g)).
	\]
	The proof for the target is similar. Also, since $T\gri(\Hor) = \Hor$, $\gri(\tau_{\gamma}^{t}(g))$ is a $\Hor$-horizontal path covering $\gri(\gamma)$ and starting at $\gri(g)$; hence the equality with the inverses.

	(c): Let $\eta$ and $h$ be as in the statement. The assumptions and the previous result implies that the paths $\tau^{t}_{\gamma}(g)$ and $\tau^{t}_{\eta}(k)$ are composable; indeed
	\[
		\grs(\tau_{\gamma}^{t}(g)) = \tau^{t}_{\grs\circ \gamma}(\grs(g)) = \tau^{t}_{\grt\circ \eta}(\grt(k)) = \grt(\tau_{\eta}^{t}(g)).
	\]
	Next, $T\grm(\Hor,\Hor) = \Hor$ implies that $\grm(\tau_{\gamma}^{t}(g), \tau_{\eta}^{t}(k))$ is a horizontal path. Since
	\[
		\pi(\grm(\tau_{\gamma}^{t}(g), \tau_{\eta}^{t}(k))) = \grm(\pi(\tau_{\gamma}^{t}(g)),\pi(\tau_{\eta}^{t}(h))) = \grm(\gamma(t),\eta(t)),
	\]
	the last equality follows.

	\eqref{item: cond MC; lift}$\implies$\eqref{item: cond MC; mult}: Suppose that the conditions in \eqref{item: cond MC; lift} hold. By Proposition~\ref{prp: cond VBsubgrpd}, we only need to show that $\Hor\tto \Hor_0$ is a subgroupoid of $T\grG\tto TM$.

	First, we show that $T\grs(\Hor), T\grt(\Hor)\subset\Hor_0$ and $T\gru(\Hor_0)\subset\Hor$. Given $v\in\Hor_{g}$, let $\tilde{\gamma}\in C^{\infty}(I,\grG)$ be a $\Hor$-horizontal curve with starting speed $v$. Remark that this is the lift of $\gamma = \pi\circ\tilde{\gamma}$ to $g$. We can conclude that
	\[
		T\grs(v) = \tfrac{d}{dt}\grs(\tilde{\gamma}(t))\big|_{t = 0} =  \tfrac{d}{dt}\grs(\tau_{\gamma}^{t}(g))\big|_{t = 0} =  \tfrac{d}{dt}\tau_{\grs(\gamma)}^{t}(\grs(g))\big|_{t = 0}\in\Hor_0.
	\]
	Similarly, we can conclude that $T\grt(\Hor)\subset\Hor_0$. Given $v\in{\Hor_0}_{x}$, take an integration $\tilde{\gamma}\in C^{\infty}(I,M)^{\Hor_0}$, and set $\gamma = \pi_0\circ\tilde{\gamma}$. Remark that $\tilde{\gamma}$ is the horizontal lift of $\gamma$ to $\tilde{\gamma}(0)$. It follows that:
	\[
		T\gru(v) = \tfrac{d}{dt}\gru(\tilde{\gamma}(t))\big|_{t = 0} =  \tfrac{d}{dt}\gru(\tau^{t}_{\gamma}(x))\big|_{t = 0} =  \tfrac{d}{dt}\tau^{t}_{\gru(\gamma)}(\gru(x))\big|_{t = 0}\in \Hor.
	\]

	Next, we show that $T\gri(\Hor)\subset\Hor$. Let $v\in\Hor_{g}$ and suppose that $\tilde{\gamma}\in C^{\infty}(I,\grG)$ is a $\Hor$-horizontal curve with starting speed $v$. Then this is the horizontal lift of $\gamma = \pi\circ\tilde{\gamma}$ to $g$ such that the following holds:
	\[
		T\gri(v) = \tfrac{d}{dt}\gri(\tilde{\gamma})\big|_{t = 0} =  \tfrac{d}{dt}\gri(\tau^{t}_{\gamma}(g))\big|_{t = 0} =  \tfrac{d}{dt}\tau^{t}_{\gri\circ\gamma}(g^{-1})\big|_{t = 0}\in\Hor.
	\]

	Lastly, let $(u,v)\in(\Hor_{g}\times\Hor_{h})\cap T\grG^{(2)}$. Since $(T\pi(u),T\pi(v))\in T\grH^{(2)}$ there is $(\gamma,\eta)\in C^{\infty}(I,\grH^{(2)})$ with initial speed $(T\pi(u),T\pi(v))$. Then
	\[
		T\grm(u,v) = \tfrac{d}{dt}\grm(\tau^{t}_{\gamma}(g),\tau^{t}_{\eta}(h))\big|_{t = 0} =  \tfrac{d}{dt}\tau^{t}_{\grm(\gamma,\eta)}(gh)\big|_{t = 0}\in\Hor.
	\]
	Therefore, \eqref{item: cond MC; mult} holds.
\end{proof}

Alternatively, we can obtain a description of multiplicative connections directly in terms of the current groupoids, which reduces to a more detailed description in the complete case. For this, we remark that the map
\[
	\pi\times \mathrm{ev}_0\colon C^{\infty}(I,\grG)\to C^{\infty}(I,\grH)\fp{\mathrm{ev}_0}{\pi}\grG,\qquad \tilde{\gamma}\mapsto (\pi\circ \tilde{\gamma},\tilde{\gamma}(0))
\]
is a groupoid morphism, where the groupoid on the right-hand side is the groupoid fibre product over $\grH$ and has base $C^{\infty}(I, N)\fp{\mathrm{ev}_0}{\pi_0}M$, covering $\pi_0\times \mathrm{ev}_0$.
\begin{proposition}\label{prp:current groupoid cond MC}
	Let $\pi\colon\grG\to\grH$ be a Lie groupoid surjective submersion covering $\pi_0\colon M\to N$, and suppose $\Hor$ and $\Hor_0$ be connections for $\pi$ and $\pi_0$, respectively. The following are equivalent:
	\begin{enumerate}
		\item \label{item:current groupoid cond MC:MEC} $\Hor\tto \Hor_0$ is a multiplicative Ehresmann connection;

		\item \label{item:current groupoid cond MC:current groupoid} $C^{\infty}(I,\grG)^{\Hor}\tto C^{\infty}(I,M)^{\Hor_0}$ is a subgroupoid of the current groupoid $C^{\infty}(I,\grG)$.
	\end{enumerate}

	In this case, the map $\pi\times\mathrm{ev}_0$ restricted to $C^{\infty}(I,\grG)^{\Hor}$ is an injective groupoid morphism, and it is an isomorphism exactly when the connection is complete. The inverse is then given by the horizontal lift:
	\[
		\tau\colon C^{\infty}(I,\grH)\fp{\mathrm{ev}_0}{\pi}\grG\diffto C^{\infty}(I,\grG)^{\Hor},\qquad (\gamma,g)\mapsto \big(t\mapsto \tau^{t}_{\gamma}(g)\big).
	\]
\end{proposition}
\begin{proof}
	\eqref{item:current groupoid cond MC:MEC}$\implies$\eqref{item:current groupoid cond MC:current groupoid}: Suppose that $\Hor\tto\Hor_0$ is a MEC. We claim that
	\begin{equation}\label{eq:base:of:groupoid}
		C^{\infty}(I,\grG)^{\Hor}\cap C^{\infty}(I,M) = C^{\infty}(I,M)^{\Hor_0}.
	\end{equation}
	Indeed, let $\gamma\in C^{\infty}(I,\grG)$. It is in $C^{\infty}(I,\grG)^{\Hor}\cap C^{\infty}(I,M)$ if and only if $\frac{d}{dt}\gamma(t)\in\Hor$ and $\gamma(t)\in M$. Since $\Hor_0 = \Hor\cap\,TM$, this is equivalent to $\frac{d}{dt}\gamma(t)\in \Hor_0$, thus to $\gamma\in C^{\infty}(I,M)^{\Hor_0}$.

	It remains to check that $C^{\infty}(I,\grG)^{\Hor}$ is closed under inversion and multiplication.

	Given a path $\gamma\in C^{\infty}(I,\grG)^{\Hor}$, we remark that:
	\[
		\tfrac{d}{dt}\gri(\gamma)(t) = T\gri\big(\tfrac{d}{dt}\gamma(t)\big)\in T\gri(\Hor) \subset \Hor,
	\]
	so that $\gri(\gamma)\in C^{\infty}(I,\grG)^{\Hor}$.

	If additionally, $\eta\in C^{\infty}(I,\grG)^{\Hor}$ such that $\grs(\gamma) = \grt(\eta)$. Then we find that:
	\[
		\tfrac{d}{dt}\grm(\gamma,\eta)(t) = T\grm\big(\tfrac{d}{dt}\gamma(t),\tfrac{d}{dt}\eta(t)\big)\in T\grm(\Hor,\Hor)\subset\Hor.
	\]
	Therefore, $C^{\infty}(I,\grG)^{\Hor}\tto C^{\infty}(I,M)^{\Hor_0}$ is a subgroupoid of the current groupoid.

	\eqref{item:current groupoid cond MC:current groupoid}$\implies$\eqref{item:current groupoid cond MC:MEC}: The assumption implies Equation~\eqref{eq:base:of:groupoid}, and it is easy to see that this equality implies that $\Hor\cap\,TM = \Hor_0$. Therefore, it suffices to show that $\Hor\subset T\grG$ is a set-theoretical subgroupoid; its base will be $\Hor_0$, and, by Proposition~\ref{prp: cond VBsubgrpd}, it will be a \VB-subgroupoid. So $\Hor\tto \Hor_0$ will be a MEC.

	Let $v\in \Hor$. Consider a path $\tilde{\gamma}\in C^{\infty}(I,\grG)^{\Hor}$ with initial speed $v$. By assumption, $\gri\circ \tilde{\gamma}\in C^{\infty}(I,\grG)^{\Hor}$. Therefore, the initial speed of this path lies in $\Hor$. We obtain that $T\gri(\Hor)\subset \Hor$.

	Finally, take $(u,v)\in (\Hor_{g}\times\Hor_{k})\cap T\grG^{(2)}$. Consider a curve $(\gamma,\eta)\colon I\to\grH^{(2)}$ with initial speed $(T\pi(u),T\pi(v))$. Denote the lifts of $\gamma$ and $\eta$ by $\tilde{\gamma}(t): = \tau_{\gamma}^{t}(g)$ and $\tilde{\eta}(t): = \tau_{\eta}^{t}(k)$ respectively. These paths are defined on $[0,\varepsilon]$, for some $\varepsilon>0$. After reparametrising all path via diffeomorphism $I\diffto [0,\varepsilon]$ with initial speed $1$, we may assume that $\tilde{\gamma},\tilde{\eta}\in C^{\infty}(I,\grG)^{\Hor}$. As we have seen above, $\grs(\tilde{\gamma})$ and $\grt(\tilde{\eta})$ are $\Hor_0$-horizontal paths, they cover the same path:
	\[
		\pi_0\circ \grs(\tilde{\gamma}) = \grs(\pi\circ \tilde{\gamma}) = \grs(\gamma) = \grt(\eta) = \grt(\pi\circ \tilde{\eta}) = \pi_0\circ \grt(\tilde{\eta}),
	\]
	and have the same initial value: $\grs(\tilde{\gamma})(0) = \grs(g) = \grt(k) = \grt(\tilde{\eta})(0)$. Hence $\grs(\tilde{\gamma}) = \grt(\tilde{\eta})$, so the paths $\tilde{\gamma}$ and $\tilde{\eta}$ are composable. Per assumption \eqref{item:current groupoid cond MC:current groupoid}, we have $\grm(\tilde{\gamma},\tilde{\eta})\in C^{\infty}(I,\grG)^{\Hor}$. Calculating the initial speed, we obtain that $T\grm(v,w)\in \Hor$.

	Clearly, the maps $\pi\colon C^{\infty}(I,\grG)\to C^{\infty}(I,\grH)$ and $\mathrm{ev}_0\colon C^{\infty}(I,\grG)\to \grG$ are groupoid morphisms, and clearly they induce a morphism to the pullback groupoid. Restricting this morphism to the subgroupoid of horizontal paths, we obtain the map from the statement. It is injective, because a horizontal path is uniquely determined by the base path and its initial value. The map is surjective, precisely when every path $\gamma\colon I\to \grH$ has a lift $\tilde{\gamma}\colon I\to \grG$ to any $g\in \pi^{-1}(\gamma(0))$. This is precisely equivalent to completeness of the connection, and clearly $\tilde{\gamma} = \tau_{\gamma}(g)$.
\end{proof}

\subsection{Infinitesimal pictures}
In the following subsections, we briefly discuss the possible counterparts of the notion of multiplicative connection in the infinitesimal setting. We will not go into more detail in this paper.

\subsubsection{Lie algebroid submersions}
Given a Lie algebroid map $\Pi\colon A\to B$, which is a surjective submersion, the notion of an \emph{infinitesimal multiplicative Ehresmann connection (IMEC)} is defined as a \VB-algebroid splitting $TA = \ker\Pi\oplus E_{\Hor}$, where $E_{\Hor}\Rightarrow \Hor_0$ is a \VB-subalgebroid of $TA\Rightarrow TM$. Given a MEC $\Hor\tto \Hor_0$ for a Lie groupoid surjective submersion $\pi\colon \grG \to \grH$, the Lie functor in the {\VB} setting \cite{BuCadH2016} yields an IMEC $E_{\Hor} = \mathrm{Lie}(\Hor)$ for the induced Lie algebroid surjective submersion $\Pi\colon A\to B$, where $A = \mathrm{Lie}(\grG)$ and $B = \mathrm{Lie}(\grH)$. Conversely, if $\grG$ is source-simply connected, then any IMEC $E_{\Hor}$ for $\Pi$ integrates to a MEC for $\pi$. This result follows from \cite{BuCadH2016}*{Section 4.3}. In the case of extensions, IMECs were studied in \cite{Fernandes2023}.

\subsubsection{Multiplicative foliations}\label{sssec:mult_fol}
Instead of a Lie groupoid surjective submersion, we could consider a multiplicative foliation $\mathcal{D}$ on a Lie groupoid $\grG\tto M$, which replaces $\ker T\pi$. The corresponding notion of a MEC is simply a \VB-groupoid complement $T\grG  = \mathcal{D}\oplus \Hor$. For extensions, this point of view was very useful in \cite{Fernandes2023} in showing that the existence of MECs is Morita invariant.

As shown in \cite{Jotz2012}, a multiplicative (simple) foliation $\mathcal{D}\subset T\grG$ does not necessarily induce a (Lie) groupoid structure on $\grG/\mathcal{D}$. Indeed assume that $\mathcal{D}$ and $\mathcal{D}_0$ are both simple, so that quotient maps $\pi\colon\grG\to\grG/\mathcal{D}$ and $\pi_0\colon M\to M/\mathcal{D}_0$ are surjective submersions. While the source, target, unit, and inverse map are well-defined on the quotients, the multiplication may fail to descend. Under the additional assumptions:
\begin{enumerate}
	\item the map
	      \[
		      \grG\to M\fp{\pi_0}{\overline{\grs}}\grG/\mathcal{D},\qquad g\mapsto (\grs(g),\pi(g))
	      \]
	      is surjective, and

	\item if $(g,h)\in \grG^{(2)}$ and $h\sim 1_{\grs(h)}$ in $\grG/\mathcal{D}$ then $hg\sim g$ in $\grG/\mathcal{D}$,
\end{enumerate}
it was shown in \cite{Jotz2012} that $\grG/\mathcal{D}\tto M/\mathcal{D}_0$ inherits a groupoid structure. In this case, $\pi$ is a Lie groupoid fibration (see Theorem~\ref{thm:weaking_fibration}).

\subsubsection{Infinitesimal ideal systems}\label{inf_ideal_systems}
To combine the settings discussed above, consider a Lie algebroid $A\Rightarrow M$ together with \emph{infinitesimally multiplicative foliation} (IMF), i.e.\ an involutive subbundle $E_{\mathcal{D}}\subset TA$ which is a \VB-subalgebroid $E_{\mathcal{D}}\Rightarrow \mathcal{D}_0$. These objects were studied in \cite{JoOr2014}, where they are also described in terms of the more explicit data of an \emph{infinitesimal ideal system}. Infinitesimally multiplicative foliations generalise the previous two settings:
\begin{itemize}
	\item if $\Pi\colon A\to B$ is a Lie algebroid submersion, then $E_{\mathcal{D}}: = \ker T\Pi$ is an IMF;

	\item if $A = \mathrm{Lie}(\grG)$, and $\mathcal{D}\subset T\grG$ is a multiplicative foliation, then $E_{\mathcal{D}}: = \mathrm{Lie}(\mathcal{D})$ is an IMF. Moreover, if $\grG$ is source-simply connected, then any IMF arises from a unique multiplicative foliation on $\grG$ \cite{JoOr2014}*{Theorem 4.8}.
\end{itemize}

An IMEC $E_{\Hor}$ in this setting is a vector bundle complement $TA = E_{\mathcal{D}}\oplus E_{\Hor}$ which is a \VB-subalgebroid $E_{\Hor}\Rightarrow \Hor_0$.

\begin{remark}
	At the infinitesimal level, there also arise notions of Lie algebroid fibrations and uniform Lie algebroid fibrations, see \cite{Mackenzie2005}*{Section~4.4}. In conjunction with Appendix~\ref{app: Lgrpd fib}, we notice, however, that at the infinitesimal level, both a Lie groupoid surjective submersion and Lie groupoid fibration define a Lie algebroid fibration. Somewhat analogously, uniform Lie groupoid fibrations descend to uniform Lie algebroid fibrations.
\end{remark}

\section{Existence}\label{sec:existence}
In this section, we discuss the problem of the existence of multiplicative Ehresmann connections in various contexts. Surjective submersions admit Ehresmann connections simply because any vector subbundle admits a complement. However, in the multiplicative setting, the existence problem is much more subtle, as \VB-subgroupoids do not always admit multiplicative complements.

\subsection{Extensions}
Consider the special case of an extension, i.e.\ a Lie groupoid surjective submersion $\pi\colon \grG\to \grH$ covering $\pi_0 = \mathrm{id}_{M}\colon M\to M$. In this case, the existence problem was well-researched already in \cite{Laurent-Gengoux2009}, where a cohomological theory was developed, and an obstruction class was identified. In \cites{Fernandes2023} and \cite{Grad2025} the problem was studied in the more general framework of \emph{bundles of ideals} on the $\grG$, which correspond to the notion of multiplicative foliations $\mathcal{D}\subset T\grG$ from Section \ref{sssec:mult_fol} with trivial base $\mathcal{D}_0 = 0_{M}$, and also in the setting of \emph{infinitesimal bundles of ideals} in Lie algebroids, which correspond to Section \ref{inf_ideal_systems} for $\mathcal{D}_0 = 0_{M}$. In particular, \cite{Grad2025} introduces a cohomological obstruction theory, both for the global level (Section 5.1.1), as well as for the infinitesimal level (Section 5.1.2). Existence was proven in \cite{Fernandes2023}*{Theorem~4.2} a bundle of ideals on a proper Lie groupoid; we recall the result in the particular case of extensions.

\begin{theorem}[\cite{Fernandes2023}]\label{thm:Fernandes}
	Any Lie groupoid extension $\pi\colon\grG\to \grH$, defined on a proper Lie groupoid $\grG$, admits a multiplicative Ehresmann connection.
\end{theorem}

This result does not extend directly to general Lie groupoid fibrations. A class of counterexamples is presented in the next subsection.

\subsection{Action morphisms}
We discuss the existence of MECs for action morphisms. This case will provide examples showing that properness is not a sufficient condition for the existence of MECs.

Consider a Lie groupoid $\grH\tto N$ with an action $a:\grH\fp{\grs}{\pi_0}M\to M$ on a surjective submersion $\pi_0:M\to N$. This yields the action groupoid $\grH\ltimes M: =  \grH\fp{\grs}{\pi_0}M\tto M$, and the action morphism
\begin{equation}\label{eq:actgrpd}
	\pi: = \mathrm{pr}_{1}:\grH\ltimes M\to \grH.
\end{equation}
Up to isomorphism, any action morphism is of this type.
\begin{proposition}
	Any MEC $\Hor\tto \Hor_0$ for \eqref{eq:actgrpd} is determined by its base connection via the equality
	\begin{equation}\label{eq:MEC:action}
		\Hor = T\mathcal{H}\fp{T\grs}{T\pi_0}\Hor_0\tto \Hor_0.
	\end{equation}
	Conversely, given an Ehresmann connection $\Hor_0\subset TM$ for $\pi_0$, the above formula defines a MEC precisely when it is $T\mathcal{H}$-invariant:
	\[
		Ta(T\mathcal{H}\fp{T\grs}{T\pi_0}\Hor_0)\subset \Hor_0.
	\]
	In particular, such a connection satisfies $\rho(\mathrm{Lie}(\grH)\fp{}{N}M))\subset \Hor_0$, where $\rho$ is denotes the infinitesimal action of $\mathrm{Lie}(\grH)$ on $\pi_0\colon M\to N$.
\end{proposition}
\begin{proof}
	The tangent of the action groupoid is also a tangent groupoid:
	\[
		T(\grH\ltimes M) = T\grH\ltimes TM.
	\]
	Therefore, if $\Hor\tto \Hor_0$ is a MEC, then $\Hor\subset T\grH\fp{T\grs}{T\pi_0}\Hor_0$. Since the two bundles have the same rank, \eqref{eq:MEC:action} follows.

	Conversely, \eqref{eq:MEC:action} defines a subgroupoid precisely when $\Hor_0$ is $T\grH$-invariant.

	The last inclusion follows from $T\grH$-invariance
	\[
		\rho(\mathrm{Lie}(\grH)\times_{N}M) = Ta(\mathrm{Lie}(\grH)\fp{T\grs}{T\pi_0}0_{M})\subset \Hor_0.\qedhere
	\]
\end{proof}

For the particular case of groups, this implies the following.

\begin{corollary}\label{cor:action}
	Let $H$ be a connected Lie group acting on a manifold $M$. The Lie groupoid fibration
	\[
		\pi: = \mathrm{pr}_{1}\colon H\ltimes M\to H,
	\]
	admits a multiplicative Ehresmann connection only if the action is trivial.
\end{corollary}
\begin{proof}
	Note that $\Hor_0 = 0_{M}$, because it is a connection on $\pi_0\colon M\to \{*\}$. The proposition implies that the infinitesimal action of $\mathrm{Lie}(H)$ on $M$ is trivial. Since $H$ is connected, so is the action of $H$ on $M$.
\end{proof}

We conclude that any non-trivial action of a compact, connected Lie group on a compact manifold yields an example of a compact Lie groupoid fibration that does not admit a MEC.

\subsection{Morita fibrations}

For this class of maps, we have a general existence result.
\begin{proposition}\label{prp:Morita_fib}
	Any Morita fibration $\pi\colon \grG\to \grH$ admits a MEC. More precisely, for any Ehresmann connection $\Hor_0$ for the base map $\pi_0:M\to N$, there is a unique MEC $\Hor$ for $\pi$ with base $\Hor_0$.
\end{proposition}
\begin{proof}
	The base map $\pi_0:M\to N$ is a surjective submersion, and we may assume that $\grG = \pi_0^{*}\grH$ and $\pi = \mathrm{pr}\colon \pi_0^{*}\grH\to \grH$, $(x,h,y)\mapsto h$.

	Let $\Hor_0\subset TM$ be an Ehresmann connection for $\pi_0$. Note that
	\[
		\overline{\pi}_0 : =  T\pi_0|_{\Hor_0}\colon \Hor_0\to TN
	\]
	is a vector bundle map and a surjective submersion. Pulling back $T\grH\tto TN$ along this map, we obtain the pullback \VB-groupoid
	\[
		\Hor : =  \overline{\pi}_0^{*}T\grH = \Hor_0\fp{\overline{\pi}_0}{T\grt}T\grH\fp{T\grs}{\overline{\pi}_0}\Hor_0\tto \Hor_0,
	\]
	which is canonically a \VB -subgroupoid of $T(\pi_0^{*}\grH)$ and defines a MEC for $\mathrm{pr}$. Uniqueness follows because any MEC with base $\Hor_0$ would be a subset of $\Hor$.
\end{proof}

\subsection{Uniform Lie groupoid fibrations}
For uniform Lie groupoid fibrations (recall Definition \ref{def:grp_fibr}), the existence problem can be reduced to the case of extensions.

\begin{proposition}\label{prp: ex MC conn uniform Lgrpd fib}
	Let $\pi\colon\grG\to\grH$ be a uniform Lie groupoid fibration. 
	If the induced map $\pi^{*}\colon \grG \to \pi_0^{*}\grH$ admits a MEC, then so does $\pi$.
\end{proposition}

\begin{proof}
	We can factor $\pi = \mathrm{pr}\circ\pi^{*}$, where $\mathrm{pr}:\pi_0^{*}\grH\to \grH$ is a Morita fibration. By Proposition \ref{prp:Morita_fib}, $\mathrm{pr}$ admits a MEC. The result follows from the lemma below.
\end{proof}

\begin{lemma}
	If the Lie groupoid surjective submersions $\pi^1\colon \grG^1\to \grG^2$ and $\pi^2\colon \grG^2\to \grG^{3}$ admit MECs, then so does their composition $\pi^2\circ\pi^1\colon \grG^1\to \grG^{3}$.
\end{lemma}
\begin{proof}
	Consider MECs $\Hor^1\tto \Hor^1_0$ for $\pi^1$ and $\Hor^2\tto \Hor^2_0$ for $\pi^2$. Then
	\[
		\Hor: = \Hor^1\cap (T\pi^1)^{-1}(\Hor^2)\tto \Hor_0: = \Hor^1_0\cap (T\pi^1_0)^{-1}(\Hor^2_0)
	\]
	is a MEC for $\pi^1\circ\pi^2$. Indeed, that $\Hor$ and $\Hor_0$ are Ehresmann connections is a general fact about Ehresmann connections for compositions of submersions. Multiplicativity follows because $T\pi^1$ is a groupoid map.
\end{proof}

Proposition \ref{prp: ex MC conn uniform Lgrpd fib} together with Theorem \ref{thm:Fernandes} yield the following.

\begin{corollary}\label{cor: proper uniform Lgrpd fib}
	Any uniform Lie groupoid fibration $\pi:\grG\to \grH$, defined on a proper Lie groupoid $\grG$, admits a multiplicative Ehresmann connection.
\end{corollary}

\section{Completeness and the connection on the kernel}\label{sec:compl}
We turn to completeness of multiplicative Ehresmann connections. We prove two completeness criteria. The first generalises \cite{Fernandes2023}*{Proposition 2.17} to fibrations and is in terms of the connections induced on the kernel, while the second is in terms of the connection induced on the base.

We discuss first the connection induced on the kernel. Given a Lie groupoid surjective submersion $\pi\colon \grG\to \grH$, its kernel
\[
	\grK: = \ker(\pi) = \pi^{-1}(\gru(N))\tto M
\]
is a family of Lie groupoids with projection denoted
\[
	\underline{\pi}: =  \grs\circ \pi|_{\mathrm{\grK}}\colon \grK\to N.
\]
First, we note that MECs are inherited by kernels (for extensions, this was remarked in \cite{Laurent-Gengoux2009}*{Section 4.1} and in \cite{Fernandes2023}*{Section 2.4}).

\begin{proposition}
	Let $\pi\colon \grG\to \grH$ be a Lie groupoid surjective submersion. Given a MEC $\Hor\tto \Hor_0$ for $\pi$, we have that
	\[
		\Hor^{\grK}: = \Hor\cap T\grK\tto \Hor_0
	\]
	is a MEC on the kernel $\underline{\pi}\colon \grK\to N$.
\end{proposition}
\begin{proof}
	We use the following general fact. Let $f:A\to B$ be a surjective submersion, $B'\subset B$ an embedded submanifold, and $A': = f^{-1}(B')$. Then $A'$ is an embedded submanifold, and $f': = f|_{A'}\colon A'\to B'$ is a surjective submersion. Moreover, if $\Hor$ is an Ehresmann connection for $f$, then $\Hor': = \Hor\cap T A'$ is an Ehresmann connection for $f'$ with corresponding horizontal lift given by restriction:
	\[
		\hor' = \hor|_{TB'\fp{}{B'}A'}\colon TB'\fp{}{B'}A'\subset TB\fp{}{B}A\to \Hor' = \Hor\cap\ TA'.
	\]

	Using this fact for $f: = \pi\colon \grG\to \grH$ and $B': = \gru(N)$, we obtain that $\Hor^{\grK}$ defined above is a connection for $\pi'\colon \grK\to \gru(N)$ (which is the map $\underline{\pi}$ from the statement, under the identification $N\cong\gru(N)$). Moreover, the horizontal lift restricts to a groupoid isomorphism:
	\[
		\hor^{\grK}\colon T\gru(N)\fp{}{\gru(N)}\grK\to \Hor^{\grK},
	\]
	where $\Hor^{\grK}$---as the intersection of two groupoids---is set-theoretically a groupoid. Using that $\hor$ is a groupoid map covering $\hor_0$, we have that:
	\[
		T\grs(\hor(T\gru(v),k)) = \hor_0(v,\grs(k)),\quad \textrm{for all}\quad v\in T_{x}N, \ k\in \pi^{-1}(\gru(x)).
	\]
	Since $\hor_0(\cdot,\grs(k))\colon T_{x}N\to (\Hor_0)_{\grs(k)}$ is a linear isomorphism, the equality shows that also $T\grs\colon (\Hor^{\grK})_{k}\to (\Hor_0)_{s(k)}$ is surjective. In particular, the base of $\Hor^{\grK}$ is $\Hor_0$. That $\Hor^{\grK}\tto \Hor_0$ is indeed a \VB-subgroupoid of $T\grK$ follows now from Proposition~\ref{prp: cond VBsubgrpd}.
\end{proof}

The following result generalises \cite{Fernandes2023}*{Proposition 2.17} to fibrations.

\begin{theorem}\label{thm: compl ker imply compl total for Lie grpd fib}
	Let $\pi\colon\grG\to\grH$ be a Lie groupoid fibration with kernel bundle $\underline{\pi}\colon \mathcal{K} := \ker(\pi)\to N$. A multiplicative Ehresmann connection $\Hor$ for $\pi$ is complete if and only if the induced multiplicative Ehresmann connection $\Hor^{\grK}$ on $\underline{\pi}$ is complete.
\end{theorem}
\begin{proof}
	Assume that $\Hor$ is a complete MEC for $\pi$. Then for any $\delta\in C^{\infty}(I,N)$ and any $k\in \pi^{-1}(\gru(\delta(0)))\subset \grK$, the $\Hor$-horizontal lift $\tau^{t}_{\gru(\delta)}(k)$ is defined on $I$. Since $\pi\circ\tau_{\gru(\delta)}(k) = \gru(\delta)$, it follows that $\tau^{t}_{\gru(\delta)}(k)\in \grK$. Thus $\frac{d}{dt}\tau^{t}_{\gru(\delta)}(k)\in T\grK\cap \Hor = \Hor^{\grK}$. So $\tau^{t}_{\gru(\delta)}(k)$ is also the $\Hor^{\grK}$-horizontal lift of $\gru(\delta)$. This shows that $\Hor^{\grK}$ is complete.

	Conversely, assume that $\Hor$ is MEC for $\pi$ such that $\Hor^{\mathcal{K}}$ is complete. Note that $\Hor$ and $\Hor^{\mathcal{K}}$ induce the same connection $\Hor_0$ on $\pi_0\colon M\to N$; simply because
	\[
		\Hor_0 = T(\gru(M))\cap \Hor\subset T\mathcal{K}\cap \Hor = \Hor^{\mathcal{K}}.
	\]
	By Proposition~\ref{prp:current groupoid cond MC} completeness of $\Hor^{\mathcal{K}}$ implies completeness of $\Hor_0$.

	Consider a smooth path $\gamma\in C^{\infty}(I,\grH)$ pick $g\in\pi^{-1}(\gamma(0))$. Assume $\gamma$ cannot be lifted completely at $g$, and consider its maximal lift
	\[
		\tau^{t}_{\gamma}(g)\colon[0,{t_{\text{m}}})\to\grG.
	\]

	By completeness of $\Hor_0$, the $\Hor_0$-lift of $\grs(\gamma)$ at $\grs(g)$ is defined on $I$, and by Proposition~\ref{prp: cond MC}, for $t\in [0,{t_{\text{m}}})$ it is given by
	\[
		\tau^{t}_{\grs(\gamma)}(\grs(g)) = \grs(\tau^{t}_{\gamma}(g)).
	\]
	In particular, we can consider $\tau^{{t_{\text{m}}}}_{\grs(\gamma)}(\grs(g))\in M$. Since $\pi$ is a Lie groupoid fibration, there exists $g'\in \grG$ such that:
	\[
		\pi(g') = \gamma({t_{\text{m}}})\qquad \textrm{and}\qquad \grs(g') = \tau^{{t_{\text{m}}}}_{\grs(\gamma)}(\grs(g)).
	\]

	Next, consider the $\Hor$-lift of $\gamma$ which at $t = {t_{\text{m}}}$ gives $g'$:
	\[
		\tau^{\cdot,{t_{\text{m}}}}_{\gamma}(g'):({t_{\text{m}}}-\delta,{t_{\text{m}}}+\delta)\to \grG,\qquad t\mapsto\tau^{t,{t_{\text{m}}}}_{\gamma}(g'),
	\]
	which is well-defined for some $0<\delta<\mathrm{min}({t_{\text{m}}},1-{t_{\text{m}}})$. Using Proposition~\ref{prp: cond MC}, this satisfies for all $t\in ({t_{\text{m}}}-\delta,{t_{\text{m}}}+\delta)$
	\[
		\grs\circ\tau^{t,{t_{\text{m}}}}_{\gamma}(g') = \tau_{\grs(\gamma)}^{t,{t_{\text{m}}}}(\grs(g')) = \tau^{t,{t_{\text{m}}}}_{\grs(\gamma)}\big( \tau^{{t_{\text{m}}}}_{\grs(\gamma)}(\grs(g))\big) = \tau^{t}_{\grs(\gamma)}(\grs(g)),
	\]
	where the last equality uses the multiplicativity property of parallel transport. We conclude that the following path is well-defined
	\[
		\eta\colon ({t_{\text{m}}}-\delta,{t_{\text{m}}})\to\grG,\qquad\eta(t): = \tau_{\gamma}^{t}(g)\cdot\big(\tau_{\gamma}^{t,{t_{\text{m}}}}(g')\big)^{-1}.
	\]
	Proposition~\ref{prp: cond MC} implies that $\eta$ is a horizontal path. It covers
	\[
		\pi(\eta(t)) = \pi(\tau_{\gamma}^{t}(g))\cdot\pi(\tau_{\gamma}^{t,{t_{\text{m}}}}(g'))^{-1} =  \gamma(t)\cdot\gamma^{-1}(t) =  \gru(\grt(\gamma))(t).
	\]
	Fix $t_0\in ({t_{\text{m}}}-\delta,{t_{\text{m}}})$ and denote $k_0: = \eta(t_0)$. By completeness of $\Hor^{\grK}$, the $\Hor^{\grK}$-horizontal lift of $\gru\circ\grt(\gamma)$, which at $t = t_0$ equals $k_0$ is defined on $I$
	\[
		\tau^{t,t_0}_{\gru\circ\grt(\gamma)}(k_0)\colon I\to \grK.
	\]
	By the uniqueness of horizontal lifts, we have that:
	\[
		\eta(t) = \tau^{t,t_0}_{\gru\circ\grt(\gamma)}(k_0), \qquad\mbox{for } t\in ({t_{\text{m}}}-\delta,{t_{\text{m}}}).
	\]
	This implies smoothness of the path
	\[
		\zeta\colon [0,{t_{\text{m}}}+ \delta)\to\grG,\qquad t\mapsto
		\begin{cases}
			\tau^{t}_{\gamma}(g)                                                             & \mbox{if }t\in [0,{t_{\text{m}}}),                            \\
			\tau^{t,t_0}_{\gru(\grt(\gamma))}(k_0)\cdot \tau^{t,{t_{\text{m}}}}_{\gamma}(g') & \mbox{if }t\in ({t_{\text{m}}}-\delta,{t_{\text{m}}}+\delta).
		\end{cases}
	\]
	The fact that $\zeta$ is horizontal and covers $\gamma$ contradicts the maximality of ${t_{\text{m}}}$.

	We conclude that $\Hor$ is complete as well.
\end{proof}

In fact, we have proven a stronger result: the implication
\[
	\Hor\textrm{ is complete }\implies \Hor^{\grK}\textrm{ is complete }
\]
does not use the assumption that $\pi$ is a Lie groupoid fibration. However, if the base is source-connected, existence of a complete MEC already implies that we are dealing with a Lie groupoid fibration, as the following result shows.

\begin{proposition}\label{prp: complt MC and s-conn imply Lgrpd fib}
	Let $\pi\colon\grG\to\grH$ be a Lie groupoid surjective submersion which admits a complete MEC. If $\grH$ is source-connected, then $\pi$ is a Lie groupoid fibration.
\end{proposition}
\begin{proof}
	By Theorem~\ref{thm:weaking_fibration} it suffices to show that, for any $x\in M$, $\pi$ restricts to a surjection between $\grs$-fibres $\pi\colon\grs^{-1}(x)\to \grs^{-1}(\pi_0(x))$. Take some $h\in \grs^{-1}(\pi_0(x))$. By assumption, there is a smooth path $\gamma\colon I\to \grs^{-1}(\pi_0(x))$ such that $\gamma(0) = \gru(\pi_0(x))$ and $\gamma(1) = h$. By completeness, $t\mapsto \tau_{\gamma}^{t}(\gru(x))$ is defined on $I$. Proposition~\ref{prp: cond MC} yields $\grs\circ\tau_{\gamma}^{t}(\gru(x)) = \tau_{\grs(\gamma)}^{t}(x) = x$. Thus $g: = \tau_{\gamma}^1(\gru(x))\in \grs^{-1}(x)$ and satisfies $\pi(g) = h$.
\end{proof}

For later use, let us note the following consequence of the proof of Theorem \ref{thm: compl ker imply compl total for Lie grpd fib}.

\begin{corollary}\label{coro:base_of_complete_is_loc_triv}
	Let $\pi\colon\grG\to\grH$ be a Lie groupoid surjective submersion which admits a complete MEC. Then $\pi_0\colon M\to N$ admits a complete Ehresmann connection.
\end{corollary}

The assumption that $\pi$ is a fibration was used for the proof of the implication
\[
	\Hor^{\grK}\textrm{ is complete }\implies \Hor\textrm{ is complete}.
\]
The next example shows that the fibration assumption was indeed necessary.

\begin{example}\label{ex: compl ker not imply compl total}
	Let $\grH = N\times \Gamma\tto N$ be a constant family of groups, where $\Gamma$ is any discrete, non-trivial group. Fix $x_0\in N$ and consider the open subgroupoid $\grH^{*} := \grH\backslash (\{x_0\}\times \Gamma\backslash\{e_{\Gamma}\})\tto N$. Consider the disjoint union Lie groupoid
	\[
		\big(\grG\tto N\sqcup N\big) :=  \big(\grH\tto N\big)\, \sqcup\, \big(\grH^{*}\tto N\big).
	\]

	We have an induced Lie groupoid surjective submersion $\pi\colon\grG\to \grH$, which is the identity and inclusion on the respective factors. The map $\pi$ is not a Lie groupoid fibration, as, for $g\neq e_{\Gamma}$, the arrow $(x_0,g)$ cannot be lifted over the copy of $x_0$ in the second factor.

	The map $\pi$ is a local diffeomorphism, and the unique connection is multiplicative. However, $\pi$ is not a covering space, as the point $(x_0,e_{\Gamma})\in \grH$ is not evenly covered, and so the connection is not complete.

	The kernel bundle is simply the unit groupoid $\grK = N\sqcup N$ and the projection $\underline{\pi}\colon\grK\to N$ is a covering space. So it admits the path-lifting property and thus the induced connection on the kernel bundle is complete.
\end{example}
Next, using the fact that the unit component of a Lie groupoid is generated by an arbitrary open neighbourhood of the units, we show that if the kernel is a source-connected Lie groupoid, then the completeness problem can be reduced further to the base connection, providing a partial converse to Corollary \ref{coro:base_of_complete_is_loc_triv}.

\begin{theorem}\label{thm: complete for s-conn}
	Let $\pi\colon\grG\to\grH$ be a Lie groupoid fibration and suppose that the kernel is source-connected. A multiplicative Ehresmann connection $\Hor$ for $\pi$ is complete if and only if the induced connection $\Hor_0$ on the base submersion $\pi_0$ is complete.
\end{theorem}

For this, we will use the following:

\begin{lemma}
	Let $X$ be a multiplicative vector field on $\grG\tto M$, which covers $X_{M}$ on $M$. If $\grG$ is source-connected and $X_{M}$ is complete, then so is $X$.
\end{lemma}
\begin{proof}
	The set $U$ where $\phi_{X}^1$ ($:= $ the flow of $X$ at time $1$) is defined, is an open subset of $\grG$. Since $X$ is multiplicative, $X$ is tangent to the zero-section $M$. Thus, its restriction to $M$ is the complete vector field $X_{M}$ and $M\subset U$. Multiplicativity implies that $U$ is a subgroupoid. The only open subgroupoid of a source-connected Lie groupoid that contains the units is the whole groupoid (see, for example \cite{Crainic2021}*{Proposition~13.7}). Hence $X$ is complete.
\end{proof}

\begin{proof}[Proof of Theorem~\ref{thm: complete for s-conn}]
	By Theorem \ref{thm: compl ker imply compl total for Lie grpd fib}, it suffices to show that completeness of $\Hor_0$ implies that of $\Hor^{\grK}$. This is equivalent to the horizontal lift to $\grK$ of any complete vector field $X\in \mathfrak{X}(N)$ be complete. Since $X$ is a multiplicative vector field on the unit groupoid, by Proposition \ref{prop:lift:multipl} $\hor^{\mathcal{K}}(X)$ is multiplicative, and it covers $\hor_0(X)$, which is complete because $\Hor_0$ is. Hence, the previous lemma yields the result.
\end{proof}

The example below shows that source-connectedness is a necessary condition.
\begin{example}\label{counter-ex: complete for s-conn}
	Consider a Lie groupoid, similar to the one in Example \ref{ex: compl ker not imply compl total}:
	\[
		\grG = (M\times \Gamma)\backslash(\{x_0\}\times \Gamma\backslash\{e_{\Gamma}\}) \tto M
	\]
	where $\Gamma$ is any discrete, non-trivial group and $x_0\in M$. Then $\grG$ is a bundle of Lie groups with projection $\pi = \mathrm{pr}_{1}\colon \grG\to M$. Notice that $\pi$ is a local diffeomorphism, and, therefore, it admits only a unique connection $\Hor = 0_{\grG}$, which is not complete as $\pi$ is not a covering map, but it is clearly multiplicative. Nevertheless, the base connection $\Hor_0 = 0_{M}$ on $\pi_0 = \mathrm{id}_{M}$ is clearly complete.
\end{example}

For Morita fibrations, by using also Proposition \ref{prp:Morita_fib}, we obtain the following.

\begin{corollary}\label{prp:Morita_fib_complete}
	A Morita fibration $\pi\colon \grG\to \grH$ admits a complete MEC if and only if the base map $\pi_0\colon M\to N$ admits a complete Ehresmann connection.
\end{corollary}
\begin{proof}
	By Corollary \ref{coro:base_of_complete_is_loc_triv}, we know that the base Ehresmann connection of a complete MEC is complete as well.

	Conversely, any Morita fibration is isomorphic to pullback $\mathrm{pr}:\pi_0^{*}\grH\to \grH$, $(x,h,y)\mapsto h$. By Proposition \ref{prp:Morita_fib}, any connection $\Hor_0$ for $\pi_0$ gives a unique MEC $\Hor$ for $\mathrm{pr}$. The explicit description of this correspondence from the proof of Proposition \ref{prp:Morita_fib} implies that the $\Hor$-horizontal lift of $\gamma\colon I\to \grH$ is given by
	\[
		t\mapsto \tau_{\gamma}^{t}(x,\gamma(0),y) = (\tau_{\grt(\gamma)}^{t}(x),\gamma(t),\tau^{t}_{\grs(\gamma)}(y)),
	\]
	for any $x\in \pi_0^{-1}(\grt(\gamma)(0))$, $y\in \pi_0^{-1}(\grt(\gamma)(0))$, where $\tau_{\grt(\gamma)}$ and $\tau_{\grs(\gamma)}$ are $\Hor_0$-horizontal lifts. Thus, if $\Hor_0$ is complete, then so is $\Hor$.
\end{proof}

Next, for uniform Lie groupoid fibrations, we obtain a variant of Proposition~\ref{prp: ex MC conn uniform Lgrpd fib} in the complete case, with the same strategy of proof.

\begin{corollary}
	Let $\pi\colon\grG\to\grH$ be a uniform Lie groupoid fibration. If $\pi_0\colon M\to N$ admits a complete Ehresmann connection and $\pi^{*}\colon\grG\to\pi_0^{*}\grH$ admits a complete MEC then $\pi$ admits a complete MEC.
\end{corollary}
The existence and completeness of multiplicative Ehresmann connections in the case of Lie groupoid extensions have been characterised in \cite{Fernandes2023}.

\section{Families of Lie groupoids}\label{sec:loc_triv_fam}
A \emph{family of Lie groupoids} is a Lie groupoid surjective submersion
\[
	\pi\colon (\grG\tto M) \to N,
\]
where $N$ is a manifold regarded as a unit groupoid. This yields indeed a family of Lie groupoids indexed by $y\in N$:
\[
	\grG_{y}: = \pi^{-1}(y)\tto \pi_0^{-1}(y),
\]
where, as usually, $\pi_0\colon M\to N$ denotes the base map of $\pi$.

In this section, we discuss in detail multiplicative Ehresmann connections for families, with focus on the existence and completeness problems.

We begin with a useful description of MECs in this setting.

\begin{lemma}\label{lema:multiplicative_lift}
	An Ehresmann connection for a family of Lie groupoids is multiplicative if and only if the horizontal vector fields coincide with the multiplicative vector fields.
\end{lemma}
\begin{proof}
	Assume that $\Hor\tto \Hor_0$ is a MEC for $\pi\colon (\grG\tto M)\to N$. Then every vector field $X\in \mathfrak{X}(N)$ is a multiplicative vector field of a unit groupoid. Therefore, Proposition \ref{prop:lift:multipl} implies that $\hor(X)$ is multiplicative.

	Conversely, let $\Hor$ be an Ehresmann connection for $\pi\colon (\grG\tto M)\to N$ such that $X_{\grG}: = \hor(X)$ is multiplicative, for all $X\in \mathfrak{X}(N)$. In particular, $X_{\grG}$ is tangent to the unit section $M$; we denote the induced vector field on $M$ by $X_{M}$. This shows that along $M$, $\Hor|_{M}\subset TM$, and so it defines an Ehresmann connection $\Hor_0: = \Hor|_{M}$ for $\pi_0$ with horizontal lift $X\mapsto X_{M}$. Since the groupoid morphism $X_{\grG}\colon \grG\to T\grG$ covers $X_{M}\colon M\to TM$, the following compatibility relations hold:
	\[
		T\grs\circ X_{\grG} = X_{M}\circ \grs\quad \textrm{and}\quad T\grt\circ X_{\grG} = X_{M}\circ \grt.
	\]
	This implies that $T\grs(\Hor)\subset \Hor_0$ and $T\grt (\Hor)\subset \Hor_0$. Consider now a composable pair $(v,w)$ in $\Hor_{g}\times \Hor_{h}$. Denote $x: = \pi(g) = \pi(h)$ and $u_0 := T_{g}\pi(v) = T_{h}\pi(w)$. We extend $u_0$ to a vector field $X$ on $N$. Horizontality implies that $v = X_{\grG}|_{g}$ and $w = X_{\grG}|_{h}$. Since $v^{\grG}$ is multiplicative, it follows that
	\[
		T_{(g,h)}\grm (u,v) = T_{(g,h)}\grm(X_{\grG}|_{g},X_{\grG}|_{h}) = X_{\grG}|_{gh}\in \Hor_{gh}.
	\]
	This shows that $\Hor$ is closed under multiplication. That $\Hor$ is also closed under inversion follows from $T \iota(X_{\grG}) = X_{\grG}\circ \iota$.
\end{proof}

Existence and completeness of multiplicative connections for families of Lie groupoids are closely related to the local triviality of the family.

\begin{definition}\label{def:loc_triv}
	The family of Lie groupoids $\pi\colon\grG\to N$ is called \emph{locally trivial} with typical fibre the Lie groupoid $\mathcal{F}\tto F$, if $N$ admits an open cover $\{U^{\alpha}\}_{\alpha\in \Lambda}$, together with local trivializations that are groupoid isomorphisms:
	\[
		\begin{tikzcd}
			{\pi^{-1}(U^{\alpha})} & {U^{\alpha}\times \mathcal{F}} \\
			{\pi_0^{-1}(U^{\alpha})} & {U^{\alpha}\times F}
			\arrow["{\psi^\alpha}", "\sim"', from = 1-1, to = 1-2]
			\arrow[shift left, from = 1-1, to = 2-1]
			\arrow[shift right, from = 1-1, to = 2-1]
			\arrow[shift left, from = 1-2, to = 2-2]
			\arrow[shift right, from = 1-2, to = 2-2]
			\arrow["{\psi^\alpha_0}", "\sim"', from = 2-1, to = 2-2]
		\end{tikzcd}, \quad \textrm{such that}\quad \mathrm{pr}_{1}\circ \psi^{\alpha} = \pi|_{\pi^{-1}(U^{\alpha})}.
	\]
\end{definition}

\subsection{Existence}
The following existence result will be refined later on to incorporate completeness as well.
\begin{proposition}\label{prp: LC imply MC}
	Any locally trivial family of Lie groupoids admits a MEC.
\end{proposition}
\begin{proof}
	Let $\{(U^{\alpha},\psi^{\alpha})\}_{\alpha\in\Lambda}$ be a trivialising cover. We use the local trivialization to obtain for each $\alpha\in \Lambda$ a MEC on $\pi^{-1}(U^{\alpha})$ with horizontal lift
	\[
		\hor^{\alpha}\colon \pi^{*}TU^{\alpha}\to T\pi^{-1}(U^{\alpha}),\qquad (x,u)\mapsto(T_{x}\psi^{\alpha})^{-1}(u,0).
	\]
	Consider a partition of unity $\{\chi^{\alpha}\}_{\alpha\in\Lambda}$ on $N$ subordinate to the cover $\{U^{\alpha}\}_{\alpha\in\Lambda}$. The pullbacks $\{ \pi^{*}\chi^{\alpha}\}_{\alpha\in \Lambda}$ define a biinvariant partition of unity on $\grG$ which is subordinate to $\{\pi^{-1}(U^{\alpha})\}_{\alpha\in\Lambda}$. We glue the horizontal lifts $\hor^{\alpha}$ using this partition of unity to obtain $\hor: =  \sum_{\alpha\in\Lambda}\pi^{*}\chi^{\alpha}\cdot \hor^{\alpha}$, a horizontal lift corresponding to a multiplicative Ehresmann connection for $\pi$.
\end{proof}
Any family of Lie groupoids $\pi\colon (\grG\tto M)\to N$ with $\pi$ proper is locally trivial (see \cite{CrMeStr20}*{Theorem 7.8} and \cite{delHoyo2019}*{Theorem 5.0.7}); therefore, it admits a MEC. The techniques developed in \cite{CrMeStr20} imply an even stronger version of this result.

\begin{theorem}\label{thm:prop_fam}
	Any family of Lie groupoids $\pi\colon (\grG\tto M)\to N$, where $\grG\tto M$ is a proper Lie groupoid, admits a MEC.
\end{theorem}

When $N$ is an interval, the result was proven in \cite{CrMeStr20}*{Lemma 6.5} using the notion of deformation cohomology, which is developed in that paper, together with explicit homotopy operators for the deformation cohomology of proper Lie groupoids. The same proof strategy applies in our case. For the convenience of the reader, we present a self-contained proof that avoids the explicit use of deformation cohomology and instead extracts only the minimal tools from \cite{CrMeStr20}. Nevertheless, \cite{CrMeStr20} remains indispensable for a deeper understanding of the phenomenon.

As in \cites{Crainic2003,CrMeStr20}, we will use that on a proper Lie groupoid $\grG\tto M$ one can build linear averaging operators along the $\grt$-fibres:
\begin{equation}\label{eq:averaging}
	C^{\infty}(\grt^{-1}(x))\to \mathbb{R},\qquad f\mapsto \int_{\grt^{-1}(x)}f(h)dh\stackrel{\mathrm{not.}}{ = }\int_{x}f(h)dh,
\end{equation}
which are normalized: $\int_{x}dh = 1$, left-invariant: $\int_{\grs(g)}f(gh)dh = \int_{\grt(g)}f(h)dh$, and smooth: if $f\in C^{\infty}(\grG)$ then $x\mapsto \int_{x}f$ is in $C^{\infty}(M)$. These can be used to make vector fields multiplicative.

\begin{lemma}
	Let $\grG\tto M$ be a proper Lie groupoid with averaging operators along the $\grt$-fibres \eqref{eq:averaging}. Let $X\in\mathfrak{X}(\grG)$ be an $\grs$-projectable vector field. Then
	\[
		g\mapsto \widehat{X}_{g}: = \int_{\grs(g)}T\grm\big(X_{gh},T\gri(X_{h})\big)dh\in T_{g}\grG,
	\]
	defines a multiplicative vector field $\widehat{X}$ on $\grG$.
\end{lemma}
\begin{proof}
	That $X$ is $\grs$-projectable is crucial for $T\grm(X_{gh},T\gri(X_{h}))$ to be defined. Let us check the multiplicativity conditions now.

	First, we calculate $\widehat{X}$ along the unit section using the inverse axiom on $T\grG$
	\[
		\widehat{X}_{\gru(x)} = \int_{x}T\grm\big(X_{h},T\gri(X_{h})\big)dh  = \int_{x}T\grt(X_{h})dh = :V_{x}\in T_{x}M.
	\]
	Hence $\widehat{X}\circ u = V$, where $V\in \mathfrak{X}(M)$ is a vector field.

	Next, we note that $\widehat{X}$ is $\grs$-projectable to $V$
	\[
		T\grs(\widehat{X}_{g}) = \int_{\grs(g)}T\grs\circ T\grm\big(X_{gh},T\gri(X_{h})\big)dh = \int_{\grs(g)}T\grt(X_{h})dh = V_{\grs(g)}.
	\]
	Using left-invariance, we obtain that $\widehat{X}$ is also $\grt$-projectable to $V$
	\begin{align*}
		T\grt(\widehat{X}_{g}) & = \int_{\grs(g)}T\grt\circ T\grm\big(X_{gh},T\gri(X_{h})\big)dh  = \int_{\grs(g)}T\grt(X_{gh})dh \\
		                       & = \int_{\grt(g)}T\grt(X_{h})dh = V_{\grt(g)}
	\end{align*}
	To show multiplicativity, note that $\grs$- and $\grt$-projectability of $\widehat{X}$ imply that, for any composable pair $(g_{1},g_{2})\in \grG^{(2)}$, the pair $(\widehat{X}_{g_1},\widehat{X}_{g_2})$ is also composable, i.e.\ it belongs to $(T\grG)^{(2)}$. Using this and left-invariance, we obtain
	\begin{align*}
		T\grm(\widehat{X}_{g_1},\widehat{X}_{g_2})
		 & = T\grm\Big(\int_{\grs(g_1)}T\grm\big(X_{g_1h},T\gri(X_{h})\big)dh,\int_{\grs(g_2)}T\grm\big(X_{g_2h},T\gri(X_{h})\big)dh\Big)       \\
		 & = T\grm\Big(\int_{\grs(g_1)}T\grm\big(X_{g_1h},T\gri(X_{h})\big)dh,\int_{\grs(g_1)}T\grm\big(X_{h},T\gri(X_{g_2^{-1}h})\big)dh\Big),
		\intertext{and using linearity of $T\grm$, we can further write}
		 & = \int_{\grs(g_1)}T\grm\Big( T\grm\big(X_{g_1h},T\gri(X_{h})\big),T\grm\big(X_{h},T\gri(X_{g_2^{-1}h})\big)\Big)dh                   \\
		 & = \int_{\grs(g_1)}T\grm\big(X_{g_1h},T\gri(X_{g_2^{-1}h})\big)dh                                                                     \\
		 & = \int_{\grs(g_2)}T\grm\big(X_{g_1g_2h},T\gri(X_{h})\big)dh = \widehat{X}_{g_1g_2}.
	\end{align*}
	Compatibility with inversion follows from the other axioms.
\end{proof}

\begin{proof}[Proof of Theorem \ref{thm:prop_fam}]
	We start with an Ehresmann connection $\Hor_0$ on $\pi_0\colon M\to N$ with horizontal lift $\hor_0\colon \mathfrak{X}(N)\to \mathfrak{X}(M)$. Consider a second Ehresmann connection $\Hor_{\grs}$ for the submersion $\grs\colon \grG\to M$ with horizontal lift $\hor_{\grs}\colon \mathfrak{X}(M)\to \mathfrak{X}(\grG)$. For $\pi$ we obtain an induced Ehresmann connection $\tilde{\Hor}: = \Hor_{\grs}\cap T\pi^{-1}(\Hor_0)$ with horizontal lift $\tilde{\hor} = \hor_{\grs}\circ \hor_0\colon \mathfrak{X}(N)\to \mathfrak{X}(\grG)$. Note that $\tilde{\hor}(X)$ is $\grs$-projectable to $\hor_0(X)$. We use the previous lemma to make the vector field $\tilde{\hor}(X)$ multiplicative, hence we define:
	\[
		\hor\colon \mathfrak{X}(N)\to \mathfrak{X}_{\mathrm{mult}}(\grG),\qquad \hor(X): = \widehat{\tilde{\hor}(X)}.
	\]
	Since $\pi\circ \grm(g,h) = \pi(g)$, we have that
	\begin{align*}
		T_{g}\pi(\hor(X))
		 & = \int_{\grs(g)}T\pi\circ T\grm\big(\tilde{\hor}(X)_{gh},T\gri(\tilde{\hor}(X)_{h})\big)dh \\
		 & = \int_{\grs(g)}T\pi(\tilde{\hor}(X)_{gh})dh = \int_{\grs(g)}X_{x}dh = X_{x},
	\end{align*}
	where $x = \pi(g) = \pi(gh)\in N$. Hence $\hor(X)$ is $\pi$-projectable to $X$. As the fibres of $\pi$ are $\grG$-invariant, the construction implies that $X\mapsto \hor(X)$ is $C^{\infty}(N)$-linear, and so it corresponds to the lift of an Ehresmann connection $\Hor$ for $\pi$. Lemma \ref{lema:multiplicative_lift} implies that $\Hor$ is multiplicative.
\end{proof}

\begin{remark}
	In \cite{CrMeStr20}*{Remark 7.2}, it is shown, by means of an example, that for a family of Lie groupoids $\pi\colon \grG \to N$, the condition that $\grG \tto M$ be a proper (or even source-proper) Lie groupoid does not imply local triviality of the family. The example, which relies on earlier work of Palais and Richardson, consists of a smooth family of actions $\rho_{t}$ of a compact Lie group $G$ on $\mathbb{R}^{n}$, depending smoothly on $t \in \mathbb{R}$. Viewed as a family of Lie groupoids $\grG = (G \ltimes \mathbb{R}^{n}) \times \mathbb{R}$ over $\mathbb{R}$, it is not locally trivial around $0$. Nevertheless, Theorem~\ref{thm:prop_fam} still yields a MEC for this family.
\end{remark}

\subsection{Completeness}
We recall Theorem \ref{thm: complete for s-conn} in the case of families.

\begin{corollary}
	For a source-connected family of Lie groupoids $\pi\colon\grG\to N$, a multiplicative Ehresmann connection $\Hor$ is complete if and only if the induced connection $\Hor_0$ on the base is complete.
\end{corollary}

\subsection{Existence of complete connections}
A surjective submersion $\pi_0\colon M\to N$ is a locally trivial fibration if and only if it admits a complete Ehresmann connection. Somewhat surprisingly, this “classical” result was only recently given a correct proof by del Hoyo \cite{delHoyo2016}.

In this subsection, we discuss the relation between the existence of complete multiplicative connections and local triviality of a family of Lie groupoids, in the sense of Definition \ref{def:loc_triv}. Clearly, if a family admits a complete MEC, then its parallel transport yields local groupoid trivializations. Conversely, if $\pi$ is a locally trivial family, then, as seen in Proposition \ref{prp: LC imply MC}, it admits a MEC---which is, however,  not guaranteed to be complete.

The main result of this section, which generalizes \cite{delHoyo2016}, is the following criterion for the existence of complete connections on locally trivial bundles.
\begin{theorem}\label{thm: s-proper implies existence of multiplicative connection}
	A locally trivial family of Lie groupoids whose typical fibre is a source-proper Lie groupoid admits a complete MEC.
\end{theorem}

The proof mimics and extends the techniques developed in \cite{delHoyo2016}. First, we give a geometric criterion for completeness of connections.

\begin{lemma}\label{lem: compl conn on fb}
	Let $p\colon A\to B$ be a locally trivial fibre bundle with typical fibre $F$, endowed with an Ehresmann connection $\Hor$. If there exists a trivialising cover
	\[
		\{\psi^{\alpha}\colon p^{-1}(U^{\alpha})\diffto U^{\alpha}\times F\}_{\alpha\in\Lambda}
	\]
	and, for each $\alpha\in \Lambda$, a closed subset $S^{\alpha}\subset F$ such that:
	\begin{enumerate}
		\item \label{thm: compl conn on fb: item: horizontal} $T\psi^{\alpha}(\Hor)|_{U^{\alpha}\times S^{\alpha}} =  TU^{\alpha}\times 0_{S^{\alpha}}$ and

		\item \label{thm: compl conn on fb: item: precmpct} the connected components of $F\backslash S^{\alpha}$ are precompact,
	\end{enumerate}
	then $\Hor$ is complete.
\end{lemma}
\begin{proof}
	Assume first that $A = B\times F$. Consider a connection $\Hor$ for the projection $\mathrm{pr}_{1}\colon B\times F\to B$, and a closed subset $S\subset F$ satisfying the above conditions. Let $\eta\colon I\to B$ be a curve. For any $x\in F$, the horizontal lift of $\eta$ starting at $(\eta(0),x)$ has the form $\tau^{t}_{\eta}(\eta(0),x) = (\eta(t),\xi(x,t))$. Note that $x\mapsto \xi(x,t)$ is the flow of the time dependent vector field $X(\cdot,t)$ on $F$ defined by the relation
	\[
		\hor_{(\eta(t),x)}\big(\tfrac{d}{dt}\eta(t)\big) = \big(\tfrac{d}{dt}\eta(t),X(x,t)\big).
	\]
	To show completeness of $\Hor$, we need to show that, for any $x\in F$, the flow line $t\mapsto \xi(x,t)$ of $X$ is defined on $I$. This follows as we are in either of two cases:
	\begin{itemize}
		\item If $x\in S$, then by (1) $X(x,t) = 0$ for all $t\in I$, so $t\mapsto \xi(x,t) = x$ is the flow line through $x$, which is defined on $I$.

		\item If $x \notin S$, then $\xi(x,t)$ stays within the connected component of $x$ in $F\backslash S$, denoted $\operatorname{Conn}(F\backslash S,x)$. Indeed, otherwise there is $t_0>0$ such that $\xi(x,t_0) = x_0\in S$, and so $t\mapsto x_0$ is a integral curve of $v$ that coincides with $\xi(x,t)$ at time $t = t_0$, therefore they must be equal, which yields the contradiction: $x = x_0\in S$. Hence, $\xi(x,t)$ must stay within the compact set $\overline{\operatorname{Conn}(F\backslash S,x)}$, which implies that it is defined on the whole of $I$.
	\end{itemize}

	The general case of a locally trivial fibre bundle follows by using the above argument along a finite cover of a path by local trivialisations.
\end{proof}

\begin{lemma}\label{prp: existence of compl MC}
	Let $\pi\colon (\grG\tto M)\to N$ be a locally trivial family of Lie groupoids with typical fibre a source-proper Lie groupoid $\grF\tto F$. Suppose that there exists
	\begin{itemize}
		\item an open cover by groupoid trivialisations $\{(V^{\alpha},\psi^{\alpha})\}_{\alpha\in\Lambda}$,
		      \[
			      \begin{tikzcd}
				      {\pi^{-1}(V^{\alpha})} & {V^{\alpha}\times \mathcal{F}} \\
				      {\pi_0^{-1}(V^{\alpha})} & {V^{\alpha}\times F}
				      \arrow["{\psi^\alpha}", "\sim"', from = 1-1, to = 1-2]
				      \arrow[shift left, from = 1-1, to = 2-1]
				      \arrow[shift right, from = 1-1, to = 2-1]
				      \arrow[shift left, from = 1-2, to = 2-2]
				      \arrow[shift right, from = 1-2, to = 2-2]
				      \arrow["{\psi^\alpha_0}", "\sim"', from = 2-1, to = 2-2]
			      \end{tikzcd},
			      \quad \textrm{with}\quad \mathrm{pr}_{1}\circ \psi^{\alpha} = \pi|_{\pi^{-1}(V^{\alpha})},
		      \]

		\item a locally finite open cover $\{U^{\alpha}\}_{\alpha\in \Lambda}$, with $\overline{U}^{\alpha}\subset V^{\alpha}$,

		\item and, for each $\alpha\in\Lambda$, a closed, $\grF$-invariant subset $S^{\alpha}\subset F$ such that
		      \begin{itemize}
			      \item the connected components of $F\backslash S^{\alpha}$ are precompact,

			      \item and, denoting $Z^{\alpha}: = (\psi_0^{\alpha})^{-1}(\overline{U}^{\alpha}\times S^{\alpha})$, such that
			            \[
				            Z^{\alpha}\cap Z^{\beta} =  \emptyset\qquad \textrm{for all}\quad \alpha\neq \beta\in \Lambda.
			            \]
		      \end{itemize}
	\end{itemize}
	Then $\pi$ admits a complete MEC.
\end{lemma}
\begin{proof}
	Clearly also $\{\overline{U}^{\alpha}\}_{\alpha\in \Lambda}$ is locally finite. Then $\{Z^{\alpha}\}_{\alpha\in \Lambda}$ is a locally finite family of $\grG$-invariant, disjoint, closed subsets in $M$. This implies that the subsets
	\[
		W^{\alpha}: =  \pi_0^{-1}(V^{\alpha})\backslash \bigcup_{\beta\neq\alpha}Z^{\beta}
	\]
	are open \cite{Munkres2000}*{Lemma 39.1 (c)}. Additionally, the family $\{W^{\alpha}\}_{\alpha\in \Lambda}$ covers $M$:
	\begin{itemize}
		\item if $x\in M\backslash \bigcup_{\beta}Z^{\beta}$, since $x$ lies in $\pi_0^{-1}(V^{\alpha})$ for some $\alpha$, we have that $x\in W^{\alpha}$;

		\item since $Z^{\alpha}\subset W^{\alpha}$, we have that $\bigcup_{\beta\in \Lambda}Z^{\beta}\subset \bigcup_{\beta\in \Lambda}W^{\beta}$.
	\end{itemize}
	Clearly, each $W^{\alpha}$ is $\grG$-invariant.

	Since $\grG\tto M$ is a proper Lie groupoid, by \cite{Crainic2017}*{Proposition 8} there exists a $\grG$-invariant partition of unity $\{\chi^{\alpha}\}_{\alpha\in \Lambda}$ on $M$ subordinate to the cover $\{W^{\alpha}\}_{\alpha\in \Lambda}$.

	For each $\alpha$, the trivialization induces a MEC on $\pi^{-1}(V^{\alpha})$ with horizontal lift
	\[
		\hor^{\alpha}\colon \pi^{*}TV^{\alpha}\to T\pi^{-1}(V^{\alpha}),\qquad \hor_{g}^{\alpha}(v) = (T_{g}\psi^{\alpha})^{-1}(v,0).
	\]
	We glue these connections using the partition of unity $\{\chi^{\alpha}\circ \grs\}_{\alpha\in \Lambda}$, which is subordinate to the cover $\{\pi^{-1}(V^{\alpha})\}_{\alpha\in \Lambda}$, and obtain a connection $\Hor$ for $\pi\colon \grG\to N$ with horizontal lift
	\[
		\hor\colon \pi^{*}TN\to T\grG,\qquad \hor = \sum_{\alpha\in \Lambda}\chi^{\alpha}\circ \grs\cdot \hor^{\alpha}.
	\]
	Since the partition of unity is $\grG$-invariant, it follows that $\Hor$ is multiplicative.

	To prove completeness of $\Hor$, we check the two assumptions of Lemma \ref{lem: compl conn on fb} for the closed subsets $\grs^{-1}(S^{\alpha})\subset \grF$.
	\begin{enumerate}
		\item Let $g\in (\psi^{\alpha})^{-1}(U^{\alpha}\times \grs^{-1}(S^{\alpha}))$. Then $\grs(g)\in Z^{\alpha}$, and so $\grs(g)\notin W^{\beta}$, for all $\beta\neq\alpha$. Therefore $\chi^{\beta}(\grs(g)) = 0$, for all $\beta\neq \alpha$. This implies that $\hor_{g} = \hor^{\alpha}_{g}$. This is equivalent to the second condition.

		\item A connected component $C$ of $\grF\backslash \grs^{-1}(S^{\alpha})$ is sent by $\grs$ into a connected component $C_0$ of $F\backslash S^{\alpha}$. Now $\grs$ is proper, and $C_0$ is precompact, so $C$ is also precompact. Hence, the first condition holds. \qedhere
	\end{enumerate}
\end{proof}

In \cite{delHoyo2016}, the subsets $S^{\alpha}$ were constructed by using an exhaustion function. We will use such a function, which is invariant, as provided by the following result.

\begin{proposition}\label{prp: s-proper = invariant exhaustion}
	Let $\grF\tto F$ be a source-proper Lie groupoid. There exists a smooth, proper function $f:F\to [0,\infty)$ which is $\grF$-invariant, i.e., $\grs^{*}f = \grt^{*}f$.
\end{proposition}
\begin{proof}
	Let $q\colon F\to F/\grF$ be the quotient map to the orbit space. By \cite{Crainic2017}*{Proposition~9}, $F/\grF$ admits a proper map $f_0\colon F/\grF\to[0,\infty)$ whose lift $f = f_0\circ q$ is smooth. Invariance of $f$ holds by construction; properness by the following.
\end{proof}

\begin{lemma}\label{lem: s-proper gives proper quotient}
	For any source-proper Lie groupoid $\grF\tto F$, the quotient map to the orbit-space $q\colon F\to F/\grF$ is proper.
\end{lemma}
\begin{proof}
	First remark that the orbit of any $x\in F$ is compact, as it is given by $\scrO_{x} =  \grt(\grs^{-1}(x))$. Next, note that any orbit $\scrO_{x}$ has pre-compact open neighbourhoods: indeed, if $\tilde{U}$ is a precompact neighbourhood of $\scrO_{x}$, then its saturation $U = \grt(\grs^{-1}(\tilde{U}))$ is an invariant precompact neighbourhood.

	Consider a compact subset $K\subset F/\grF$. For any $\scrO\in K$ pick a precompact, invariant neighbourhood $\scrO\subset U_{\scrO}\subset M$. The open cover $\{U_{\scrO}/\grF\}_{\scrO\in K}$ admits a a finite open subcover $\{U_{\scrO_i}/\grG\}_{i = 1}^{n}$. Then the closed set $q^{-1}(K)$ admits the precompact neighbourhood $U: = \cup_{i = 1}^{n}U_{\scrO_i}$, and therefore it is compact.
\end{proof}

We are now ready to conclude:
\begin{proof}[Proof of Theorem \ref{thm: s-proper implies existence of multiplicative connection}]
	Let $\grF\tto F$ be the typical fibre of $\pi\colon \grG\to N$. Consider a function $f:F\to [0,\infty)$ as in Proposition~\ref{prp: s-proper = invariant exhaustion}.

	We will construct now all the ingredients needed in Lemma \ref{prp: existence of compl MC}.

	Consider a locally finite groupoid trivialising cover for $\pi\colon \grG\to N$ made of precompact subsets $\{(V^{\alpha},\psi^{\alpha})\}_{\alpha\in\Lambda}$, and $\{U^{\alpha}\}_{\alpha\in\Lambda}$ a second open cover, with $\overline{U}^{\alpha}\subset V^{\alpha}$.

	The sets $S^{\alpha}$ will be constructed of the form:
	\[
		S^{\alpha}: = f^{-1}\big(\{n(i,\alpha)\, |\, i\in \mathbb{N}\}\big),
	\]
	where $i\mapsto n(i,\alpha)\in \mathbb{N}$ is an increasing sequence of integers. This ensures that each $S^{\alpha}$ is $\grF$-invariant. Moreover, for any $x\in F$, there is $i\in \mathbb{N}$ such that $f(x)< n(i,\alpha)$. Then the connected component of $x$ in $F\backslash S^{\alpha}$ is included in $f^{-1}([0,n(i,\alpha)])$, and hence, by properness of $f$, it is precompact.

	To ensure the last condition in Lemma \ref{prp: existence of compl MC}, we construct $n(i,\alpha)$ so that the sets
	\[
		Z^{i,\alpha}: = (\psi^{\alpha}_0)^{-1}\big( \overline{U}^{\alpha}\times f^{-1}(n(i,\alpha))\big)\subset \pi_0^{-1}(\overline{U}^{\alpha})
	\]
	are pairwise disjoint, i.e., for $(i,\alpha)\neq (j,\beta)$,
	\begin{equation}\label{eq:disjoint}
		Z^{i,\alpha}\cap Z^{j,\beta} = \emptyset.
	\end{equation}
	Since $Z^{\alpha} = \bigcup_{i\in \mathbb{N}}Z^{i,\alpha}$, the required condition follows: $Z^{\alpha}\cap Z^{\beta} = \emptyset$, for $\alpha\neq \beta$.

	Without loss of generality, we can assume that $\Lambda = \bbN$ or $\Lambda = \{0,\ldots,N\}$. We consider the lexicographic ordering on $\bbN\times\Lambda$: $(j,\beta)\leq (i,\alpha)$ if and only if $j < i$, or $j = i$ and $\beta\leq \alpha$. We use induction with respect to this well-ordering to construct the map $n\colon\bbN\times\Lambda\to\bbN$. Set $n(0,0) = 0$. Suppose that $n$ is defined on all $(j,\beta)$, with $(j,\beta)<(i,\alpha)$. Define $n(i,\alpha)$ so that
	\[
		n(i,\alpha)> \sup \big\{f\circ \mathrm{pr}_{2}\circ \psi_0^{\alpha}(x)\, |\, x\in \pi_0^{-1}(\overline{U}^{\alpha})\cap Z^{j,\beta}, \ (j,\beta)<(i,\alpha)\big\}.
	\]
	We note that the right-hand side is indeed bounded. Indeed, $\{\overline{U}^{\gamma}\}_{\gamma\in\Lambda}$ is a locally finite collection of compact sets, so only finitely many $\beta\in\Lambda$ satisfy $\overline{U}^{\alpha}\cap \overline{U}^{\beta}\neq\emptyset$, and so only finitely many pairs $(j,\beta)$ with $(j,\beta)<(i,\alpha)$ satisfy $\pi_0^{-1}(\overline{U}^{\alpha})\cap Z^{j,\beta}\neq \emptyset$. Hence, the supremum is taken over a compact set.

	The choice of $n(i,\alpha)$ ensures \eqref{eq:disjoint}. Indeed, if $y\in Z^{i,\alpha}\cap Z^{j,\beta}$, we obtain
	\[
		f\circ \mathrm{pr}_{2}\circ \psi_0^{\alpha}(y) = n(i,\alpha)>\mathrm{sup}\{f\circ \mathrm{pr}_{2}\circ \psi_0^{\alpha}(x)\,|\, x\in \pi_0^{-1}(\overline{U}^{\alpha})\cap Z^{j,\beta}\}\geq f\circ \mathrm{pr}_{2}\circ \psi_0^{\alpha}(y),
	\]
	which is a contradiction.

	Theorem \ref{thm: s-proper implies existence of multiplicative connection} follows now by applying Lemma \ref{prp: existence of compl MC}.
\end{proof}

It would be interesting to understand whether source-properness is indeed a necessary condition for the existence of complete multiplicative connections.

\begin{question}
	Does every locally trivial family of Lie groupoids admit a complete multiplicative Ehresmann connection?
\end{question}

\appendix

\section{Lie groupoid fibrations}\label{app: Lgrpd fib}
Lie groupoid fibrations (recall Definition \ref{def:grp_fibr}) were studied before; for the non-smooth setting see \cites{Brown1970, Chen2020}, and for the smooth setting see \cites{Mackenzie2005, delHoyo2019}. In this appendix, we provide simpler criteria for a functor to be a Lie groupoid fibration. We apply these to show that conclude that uniform Lie groupoid fibrations are indeed Lie groupoid fibrations.

\begin{theorem}\label{thm:weaking_fibration}
	Let $\pi$ be a Lie groupoid morphism $\pi\colon \grG \to \grH$ with base map $\pi_0\colon M\to N$, and denote
	\[
		\pi^{!}: = \pi\times \grs\colon\grG \to \grH\fp{s}{\pi_0}M.
	\]
	Any of the following conditions is equivalent to $\pi$ being a Lie groupoid fibration.
	\begin{enumerate}
		\item\label{item: weaking_fibration; full} The maps $\pi_0$ and $\pi^{!}$ are surjective submersions.

		\item\label{item: weaking_fibration; shriek-sur} The map $\pi$ is a surjective submersion and the map $\pi^{!}$ is surjective.

		\item\label{item: weaking_fibration; star-sur} The map $\pi$ is a surjective submersion, and it restricts to a surjective map between source-fibres, i.e., $\pi\colon \grs^{-1}(x)\to \grs^{-1}(\pi_0(x))$ is surjective for any $x\in M$.
	\end{enumerate}
\end{theorem}

In the proof, we will use some auxiliary results.

\begin{lemma}\label{lem: sur_sum bs}
	Suppose $f\colon X\to Y$ is a surjective submersion. Then for any smooth map $g\colon Z\to Y$, the base change map $\widehat{f}\colon Z\fp{g}{f}X\to Z$ is a surjective submersion.
\end{lemma}
\begin{proof}
	Let $z\in Z$, and let $y: = g(z)$. Since $f$ is surjective, there exists $x\in X$ such that $f(x) = y$. In other words $(z,x)\in Z\fp{g}{f}X$ and $\widehat{f}(z,x) = z$. So $\widehat{f}$ is surjective. Moreover, since $f$ is a submersion, for any $x\in f^{-1}(y)$ we find a smooth local section $\sigma:U\to X$ of $f$, i.e., $f\circ \sigma = \mathrm{id}_{U}$, such that $\sigma(y) = x$. Then the map
	\[
		\widehat{\sigma}\colon g^{-1}(U)\to Z\fp{g}{f}X,\qquad w\mapsto (w,\sigma(g(w)))
	\]
	is a smooth local section of $\widehat{f}$ with $\widehat{\sigma}(z) = (z,x)$. Hence $\widehat{f}$ is a submersion.
\end{proof}

\begin{lemma}\label{lem: sur in vspb}
	Consider a commutative diagram of vector spaces
	\[
		\begin{tikzcd}
			U & Y \\
			X & Z
			\arrow["F", from = 1-1, to = 1-2]
			\arrow["G", from = 1-1, to = 2-1]
			\arrow["g", from = 1-2, to = 2-2]
			\arrow["f", from = 2-1, to = 2-2]
		\end{tikzcd}
	\]
	Suppose that $F$ is surjective and the restriction $G\colon\ker F\to\ker f$ is surjective, then $G\times F\colon U\to X\fp{f}{g}Y$ is surjective.
\end{lemma}
\begin{proof}
	Take some $(x,y)\in X\fp{f}{g}Y$ and notice that by surjectivity of $F$ we can find $u_{1}\in U$ such that $F(u_{1}) = y$. It follows that
	\[
		f(x - G(u_{1})) = f(x) - f(G(u_{1})) = f(x) - g(F(u_{1})) = f(x) - g(y) = 0.
	\]
	Therefore, $x - G(u_{1})\in \ker f$, and so there exists $u_{2}\in \ker F$ for which $G(u_{2}) = x - b(u_{1})$. We define $u = u_{1}+ u_{2}$ and remark that
	\[
		F(u_{1}+ u_{2}) = F(u_{1}) = y,\quad G(u_{1}+ u_{2}) = x\implies G\times F(u) = (x,y).
	\]
	Therefore, $G\times F$ is surjective.
\end{proof}
\begin{lemma}\label{lem: sur to pb}
	Consider a commutative diagram of sets
	\[
		\begin{tikzcd}
			U & Y \\
			X & Z
			\arrow["F", from = 1-1, to = 1-2]
			\arrow["G", from = 1-1, to = 2-1]
			\arrow["g", from = 1-2, to = 2-2]
			\arrow["f", from = 2-1, to = 2-2]
		\end{tikzcd}
	\]
	Then the following are equivalent:
	\begin{enumerate}
		\item \label{item: sur to pb; ind map} The map $G\times F\colon U\to X\fp{f}{g}Y$ is surjective.

		\item \label{item: sur to pb; fw} For all $z\in Z$ and all $x\in f^{-1}(z)$, the map $F\colon G^{-1}(x)\to g^{-1}(z)$ is surjective.
	\end{enumerate}
\end{lemma}
\begin{proof}
	\eqref{item: sur to pb; ind map}$\implies$\eqref{item: sur to pb; fw}: Suppose $G\times F\colon U\to X\fp{f}{g}Y$ is surjective and pick some arbitrary $z\in Z$ and $x\in f^{-1}(z)$. For $y\in g^{-1}(z)$ we have $(x,y)\in X\fp{f}{g}Y$ such that there exists some $u\in U$ with $G\times F(u) = (x,y)$, per assumption. This implies that $F(u) = y\in g^{-1}(z)$ and $G(u) = x$ such that $u\in G^{-1}(x)$. Therefore, the restriction $F\colon G^{-1}(x)\to g^{-1}(z)$ is surjective.

	\eqref{item: sur to pb; fw}$\implies$\eqref{item: sur to pb; ind map}: Suppose that for all choices $z\in Z$ and $x\in f^{-1}(z)$ the map $F\colon G^{-1}(x)\to g^{-1}(z)$ is surjective. Fix some $(x,y)\in X\fp{f}{g}Y$. This implies that $x\in f^{-1}(g(y))$ and as $F\colon G^{-1}(x)\to g^{-1}(g(y))$ is surjective, we find some $u\in G^{-1}(x)$ with $F(u) = y$. Therefore, $G\times F(u) = (x,y)$
\end{proof}

\begin{proof}[Proof of \ref{thm:weaking_fibration}.\eqref{item: weaking_fibration; full}]
	Condition \eqref{item: weaking_fibration; full} is part of the definition of a Lie groupoid fibration. To show that the condition is sufficient, we need to check that it implies that $\pi$ is a surjective submersion. Applying Lemma \ref{lem: sur_sum bs} to $f = \pi_0\colon M\to N$ and $g = \grs\colon\grH\to N$, it follows that the induced map $\widehat{\pi}_0\colon \grH\fp{\grs}{\pi_0}M\to \grH$ is a surjective submersion. The decomposition $\pi = \widehat{\pi}_0\circ\pi^{!}$ shows that $\pi$ is a surjective submersion.
\end{proof}

\begin{proof}[Proof of \ref{thm:weaking_fibration}.\eqref{item: weaking_fibration; shriek-sur}]
	As condition \eqref{item: weaking_fibration; shriek-sur} is part of the definition of fibration, we only need to show that it implies that $\pi^{!}$ is a submersion. Let $g\in \grG$ and denote $(h,x): = \pi^{!}(g) = (\pi(g),\grs(g))$. By assumption, $T_{g}\pi:T_{g}\grG\to T_{h}\grH$ is surjective. On the other hand, by Proposition~\ref{prp: cond VBsubgrpd} (or \cite{LiBland2014}*{Corollary~C.3}), we have that $\ker T\pi\tto \ker T\pi_0$ is a \VB-subgroupoid. In particular, $T_{g}\grs\colon \ker T_{g}\pi\to \ker T_{x}\pi_0$ is surjective. Therefore, we conclude by Lemma~\ref{lem: sur in vspb} that $T_{g}\pi^{!}$ is surjective. So $\pi^{!}$ is a submersion.
\end{proof}

\begin{proof}[Proof of \ref{thm:weaking_fibration}.\eqref{item: weaking_fibration; star-sur}]
	Lemma \ref{lem: sur to pb} shows that $\pi^{!}$ is surjective exactly when $\pi:\grs^{-1}(x)\to \grs^{-1}(\pi_0(x))$ is surjective for all $x\in M$. So items (b) and (c) are equivalent.
\end{proof}

Example \ref{ex: compl ker not imply compl total} shows that Lie groupoid fibrations indeed define a proper subclass of surjective submersions of Lie groupoids.

\begin{remark}
	\label{remark_fibrations} We remark that in the set-theoretical setting, the source-fibrewise surjectivity of the map is used to define Lie groupoid fibrations, and this property is also called a star-surjective groupoid morphism (see, e.g., \cite{Brown1970} and \cite{Chen2020}).
\end{remark}

With all the tools in place, it is easy to obtain that Lie groupoid fibrations are closed under taking pullbacks---much like surjective submersions or covering spaces.

\begin{corollary}\label{prp: bc of Lgrpd fib}
	Let $\pi\colon \grG \to \grH$ be a Lie groupoid fibration and $\phi\colon\grK\to \grH$ be a Lie groupoid morphism. Then $\widehat{\pi}:\grK\fp{\phi}{\pi}\grG\to\grK$ is also a Lie groupoid fibration.
\end{corollary}
\begin{proof}
	Lemma \ref{lem: sur_sum bs} shows that $\widehat{\pi}$ is a surjective submersion. By Theorem~\ref{thm:weaking_fibration} (c), we need only show that $\widehat{\pi}$ is surjective between source fibres. Denote the bases of the groupoid $\grG$, $\grH$ and $\grK$ by $M$, $N$ and $K$, respectively. Pick some $(x,y)\in K\fp{\phi_0}{\pi_0}M$ and $k\in \grs^{-1}(x)\subset \grK$. Note that
	\[
		\grs(\phi(k)) = \phi_0(\grs(k)) = \phi_0(x) = \pi_0(y).
	\]
	So, $(\phi(k),y)\in \grH\fp{\grs}{\pi_0}M$. As $\pi$ is a Lie groupoid fibration, there exists some $g\in \grG$ such that $\pi^{!}(g) = (\pi(g),\grs(g)) = (\phi(k),y)$. So $(k,g)\in \grK\fp{\phi}{\pi}\grG$, $\widehat{\pi}(k,g) = k$ and $\grs((k,g)) = (x,y)$. This shows that $\widehat{\pi}$ is surjective between source-fibres.
\end{proof}

Next, we show that the name \emph{uniform Lie groupoid fibration} is indeed justified.

\begin{corollary}
	\label{prp: uniform fib is fib} Any uniform Lie groupoid fibration is a Lie groupoid fibration.
\end{corollary}
\begin{proof}
	Suppose that $\pi\colon \grG\to \grH$ is a uniform Lie groupoid fibration. Note that the map $\pi^{*}\colon \grG\to \pi_0^{*}\grH$ is a Lie groupoid extension, and thus, in particular, a Lie groupoid fibration. Additionally, as $\pi_0\times\pi_0\colon M\times M\to N\times N$ is a Lie groupoid fibration, by Corollary \ref{prp: bc of Lgrpd fib}, so is the base-change $\widehat{\pi_0\times\pi_0}\colon \pi_0^{*}\grH \to\grH$. This implies that $\pi = \widehat{\pi_0\times\pi_0}\circ\pi^{*}$ is also a Lie groupoid fibration.
\end{proof}

Next, we show that uniform fibrations are indeed a proper subclass of fibrations.
\begin{example}
	\label{counter-ex: fib not uniform} Let $\grH\tto N$ be a Lie groupoid, $P$ a manifold of $\dim P>0$, and denote their product groupoid as $\grG = \grH\times P\tto N\times P$. The projection map $\mathrm{pr}_{1}\colon\grG\to\grH$ defines a Lie groupoid fibration; it is the pullback of the trivial fibration $P\to \{*\}$ along the trivial groupoid map $\grH\to\{*\}$. However, it is not uniform. Indeed
	\[
		\pi^{*}\colon \grG  = \grH\times{P}\to \pi_0^{*}(\grH) = \grH\times{P}\times P, \quad (h,p)\mapsto (h,p,p)
	\]
	is neither surjective nor a submersion.
\end{example}

\section{\VB-subgroupoids}
\label{app: proofs}

The following statement characterizes \VB-subgroupoids as subgroupoids that are vector bundles, and has been used several times in the paper. The result generalizes the first part of \cite{Jotz2012}*{Lemma~3.5}, and is proven similarly.

\begin{proposition}
	\label{prp: cond VBsubgrpd} Let $(\Gamma, E \sslash \grG,M)$ be a \VB-groupoid. A vector subbundle $S\subset \Gamma$ over $\grG$ which is a subgroupoid is a \VB-subgroupoid of $\Gamma$ with base $S|_{M}\cap E$.
\end{proposition}

\begin{remark}
	In the statement, we allow the vector bundle $S|_{M}\cap E$ to have possibly different ranks over different connected components of $M$. In fact, from the proof one can easily see that the rank of $S|_{M}\cap E$ will be the same over two connected components $M_{1}$ and $M_{2}$ that are related by an arrow $g\in \grG$ with $\grs(g)\in M_{1}$ and $\grt(g)\in M_{2}$.
\end{remark}

In the following auxiliary result, we also allow $V_{1}\cap V_{2}$ to have different ranks over different connected components.

\begin{lemma}\label{lem:subbundle}
	\label{lem:int_subb} Let $V$ be a vector bundle and $V_{1},V_{2}\subset V$ vector subbundles, all over the connected manifold $M$. Then $V_{1}\cap V_{2}$ is a vector subbundle if and only if each $v\in V_{1}\cap V_{2}$ can be extended to a smooth local section of $V$ with values in $V_{1}\cap V_{2}$.
\end{lemma}
\begin{proof}
	As $V_{1}\cap V_{2}$ is the kernel of the vector bundle map $V_{1}\to V/V_{2}$, it is a vector subbundle if and only if it has locally constant rank. As a kernel of a vector bundle map, its rank can only drop locally. The property that any $v\in V_{1}\cap V_{2}$ can be extended to a local section implies that the rank can only grow locally, and thus, it implies that $V_{1}\cap V_{2}$ is a vector subbundle. The converse is obvious.
\end{proof}

\begin{proof}[Proof of Proposition \ref{prp: cond VBsubgrpd}]
	We divide the proof into several steps.

	\emph{Step 1: $S|_{M}\cap E$ is a vector subbundle of $E$.}

	Here we identify $E$ with the unit section $E\equiv\gru(E)\subset \Gamma$, which is a vector subbundle over the unit section of $\grG$, which we also identify as $M\equiv\gru(M)\subset \grG$. So $S|_{M}$ and $E$ are vector subbundles of $\Gamma|_{M}$. We check the criterion in Lemma \ref{lem:int_subb} for each connected component of $M$. Let $e\in (S\cap E)_{x}$. Pick a local section of $S|_{M}$, $\sigma\colon U\subset M\to S|_{M}$, such that $\sigma(x) = e$. Then $\tilde{\grs}\circ\sigma\colon U\to E$ is a smooth section of $E$ with values in $S|_{M}\cap E$ such that $\tilde{\grs}\circ\sigma(x) = \tilde{\grs}(e) = e$.

	\emph{Step 2: $S|_{M}\cap C$ is a vector bundle, where $C: = \ker\tilde{s}|_{M}$ is the core of $\Gamma$.}

	Notice that $\Gamma|_{M} =  C\oplus E$, with projection to $E$ given by $\tilde{\grs}\colon\Gamma|_{M}\to E$ and inclusion of $E$ given by $\tilde{\gru}\colon E\to \Gamma|_{M}$. Since $S$ is a groupoid, it is invariant under these two maps, which shows that $S|_{M} =  (S|_{M}\cap C)\oplus (S|_{M}\cap E)$. Using Step 1, it follows that $S|_{M}\cap C$ has locally constant rank, which implies the claim.

	\emph{Step 3: $S\cap\ker\tilde{\grs}$ is a vector subbundle.}

	We first note that $\ker\tilde{\grs}$ is canonically isomorphic to $\grt^{*}C$ via the morphism:
	\[
		j\colon \grt^{*}C\to \ker\tilde{\grs},\qquad (c,g)\mapsto c\cdot0_{g},\quad c\in C_{\grt(g)};
	\]
	in other words, right translation by $0_{g}\in \Gamma_{g}$, denoted as $r_{0_g}\colon C_{\grt(g)}\to \ker\tilde{\grs}_{g}$, is a linear isomorphism. Now as $S$ is a subgroupoid of $\Gamma$, this map restricts to an isomorphism $r_{0_g}\colon S_{\grt(g)}\cap C_{\grt(g)}\to S_{g}\cap \ker\tilde{\grs}_{g}$. So $S\cap\ker\tilde{\grs}$ has locally constant rank. By Lemma~\ref{lem:subbundle}, it follows that it is a vector subbundle.

	\emph{Step 4: $\tilde{\grs}\colon S\to S|_{M}\cap E$ is a surjective submersion.}

	First, the previous steps imply that the map is fibrewise surjective:
	\[
		\mathrm{rank}(S|_{M}\cap E) = \mathrm{rank}(S|_{M})-\mathrm{rank}(S|_{M}\cap C) = \mathrm{rank}(S)-\mathrm{rank}(S\cap \ker\tilde{\grs}) = \mathrm{rank}(\tilde{\grs}|_{S}).
	\]
	Since the map covers the surjective submersion $\grs$, it is a surjective submersion as well.

	\emph{Step 5: Conclusion.}

	We can conclude that $S\tto S|_{M}\cap E$ is a Lie groupoid. Being contained in $\Gamma\tto E$, it is compatible with the vector bundle structure, so as to be a \VB-groupoid.
\end{proof}

\bibliography{Bibliography}

@article{Crainic2003,
  author  = {Crainic, Marius},
  title   = {Differentiable and algebroid cohomology, Van Est isomorphisms, and characteristic classes},
  journal = {Commentarii Mathematici Helvetici},
  volume  = {78},
  number  = {4},
  pages   = {681--721},
  year    = {2003},
  doi     = {10.1007/s00014-003-0775-2}
}

@article{MackXu98,
    AUTHOR = {Mackenzie, Kirill C. H. and Xu, Ping},
     TITLE = {Classical lifting processes and multiplicative vector fields},
   JOURNAL = {Quart. J. Math. Oxford Ser. (2)},
  FJOURNAL = {The Quarterly Journal of Mathematics. Oxford. Second Series},
    VOLUME = {49},
      YEAR = {1998},
    NUMBER = {193},
     PAGES = {59--85},
      ISSN = {0033-5606,1464-3847},
   MRCLASS = {58H05 (58F05)},
  MRNUMBER = {1617335},
MRREVIEWER = {Charles-Michel\ Marle},
       DOI = {10.1093/qjmath/49.193.59},
       URL = {https://doi.org/10.1093/qjmath/49.193.59},
}

@article{Brown1970,
  author   = {Brown, Ronald},
  doi      = {10.1016/0021-8693(70)90089-X},
  fjournal = {Journal of Algebra},
  issn     = {0021-8693},
  journal  = {J. Algebra},
  pages    = {103--132},
  title    = {Fibrations of groupoids},
  volume   = {15},
  year     = {1970},
}

@article{BuCadH2016,
  author     = {Bursztyn, Henrique and Cabrera, Alejandro and del Hoyo,
                Matias},
  doi        = {10.1016/j.aim.2015.11.044},
  fjournal   = {Advances in Mathematics},
  issn       = {0001-8708,1090-2082},
  journal    = {Adv. Math.},
  mrclass    = {53D17},
  mrnumber   = {3451921},
  mrreviewer = {Iakovos\ Androulidakis},
  pages      = {163--207},
  title      = {Vector bundles over {L}ie groupoids and algebroids},
  url        = {https://doi.org/10.1016/j.aim.2015.11.044},
  volume     = {290},
  year       = {2016},
}

@article{Chen2020,
  author   = {Chen, Bohui and Du, Cheng-Yong and Wang, Yu},
  doi      = {10.1016/j.geomphys.2020.103644},
  fjournal = {Journal of Geometry and Physics},
  issn     = {0393-0440},
  journal  = {J. Geom. Phys.},
  pages    = {27},
  title    = {On fibrations of {Lie} groupoids},
  volume   = {152},
  year     = {2020},
}

@article{Crainic_Salazar,
  author     = {Crainic, Marius and Salazar, Maria Amelia and Struchiner,
                Ivan},
  doi        = {10.1007/s00209-014-1398-z},
  fjournal   = {Mathematische Zeitschrift},
  issn       = {0025-5874,1432-1823},
  journal    = {Math. Z.},
  mrclass    = {58H05 (53C05 58H10)},
  mrnumber   = {3318255},
  mrreviewer = {I.\ Kol\'{a}\v{r}},
  number     = {3-4},
  pages      = {939--979},
  title      = {Multiplicative forms and {S}pencer operators},
  url        = {https://doi.org/10.1007/s00209-014-1398-z},
  volume     = {279},
  year       = {2015},
}

@article{Crainic2017,
  author   = {Crainic, Marius and Mestre, Jo{\~a}o Nuno},
  doi      = {10.1007/s11005-017-1011-6},
  fjournal = {Letters in Mathematical Physics},
  issn     = {0377-9017},
  journal  = {Lett. Math. Phys.},
  number   = {3},
  pages    = {805--859},
  title    = {Orbispaces as differentiable stratified spaces},
  volume   = {108},
  year     = {2018},
}

@book{Crainic2021,
  author    = {Crainic, Marius and Fernandes, Rui Loja and Mărcuț, Ioan},
  doi       = {10.1090/gsm/217},
  fseries   = {Graduate Studies in Mathematics},
  isbn      = {978-1-4704-6667-1; 978-1-4704-6666-4},
  issn      = {1065-7339},
  publisher = {Providence, RI: American Mathematical Society (AMS)},
  series    = {Grad. Stud. Math.},
  title     = {Lectures on {Poisson} geometry},
  volume    = {217},
  year      = {2021},
}

@article{CrMeStr20,
  author     = {Crainic, Marius and Mestre, Jo\~{a}o Nuno and Struchiner,
                Ivan},
  doi        = {10.1093/imrn/rny221},
  fjournal   = {International Mathematics Research Notices. IMRN},
  issn       = {1073-7928,1687-0247},
  journal    = {Int. Math. Res. Not. IMRN},
  mrclass    = {58H15 (58H05)},
  mrnumber   = {4176835},
  mrreviewer = {Tom\'{a}\v{s}\ Gedeon},
  number     = {21},
  pages      = {7662--7746},
  title      = {Deformations of {L}ie groupoids},
  url        = {https://doi.org/10.1093/imrn/rny221},
  year       = {2020},
}

@article{delHoyo2016,
  author   = {del Hoyo, Matias},
  doi      = {10.1016/j.indag.2016.06.009},
  fjournal = {Indagationes Mathematicae. New Series},
  issn     = {0019-3577},
  journal  = {Indag. Math., New Ser.},
  number   = {4},
  pages    = {985--990},
  title    = {Complete connections on fiber bundles},
  volume   = {27},
  year     = {2016},
}

@article{delHoyo2019,
  author   = {del Hoyo, Matias and Fernandes, Rui Loja},
  doi      = {10.1007/s00209-018-2154-6},
  fjournal = {Mathematische Zeitschrift},
  issn     = {0025-5874},
  journal  = {Math. Z.},
  number   = {1-2},
  pages    = {103--132},
  title    = {Riemannian metrics on differentiable stacks},
  volume   = {292},
  year     = {2019},
}

@article{Fernandes2023,
  author   = {Fernandes, Rui Loja and Mărcuț, Ioan},
  doi      = {10.1016/j.aim.2023.109124},
  fjournal = {Advances in Mathematics},
  issn     = {0001-8708},
  journal  = {Adv. Math.},
  pages    = {84},
  title    = {Multiplicative {Ehresmann} connections},
  volume   = {427},
  year     = {2023},
}

@article {FerMa26,
    AUTHOR = {Fernandes, Rui Loja and Mărcuț, Ioan},
     TITLE = {Poisson geometry around {P}oisson submanifolds},
   JOURNAL = {J. Eur. Math. Soc. (JEMS)},
  FJOURNAL = {Journal of the European Mathematical Society (JEMS)},
    VOLUME = {28},
      YEAR = {2026},
    NUMBER = {3},
     PAGES = {1213--1311},
      ISSN = {1435-9855,1435-9863},
   MRCLASS = {53D17 (58H05)},
  MRNUMBER = {5032976},
       DOI = {10.4171/jems/1494},
       URL = {https://doi.org/10.4171/jems/1494},
}

@misc{Grad2025,
  arxiv        = {arXiv:2503.08873},
  author       = {Grad, {\v{Z}}an},
  howpublished = {Preprint, {arXiv}:2503.08873 [math.{DG}] (2025)},
  title        = {Covariant derivatives in the representation-valued {Bott}-{Shulman}-{Stasheff} and {Weil} complex},
  url          = {https://arxiv.org/abs/2503.08873},
  year         = {2025},
}

@phdthesis{GradPhD,
  author = {Grad, {\v{Z}}an},
  school = {Instituto Superior T\'{e}cnico Lisbon},
  title  = {Fundamentals of {L}ie categories and {Y}ang--{M}ills theory for multiplicative {E}hresmann connections},
  year   = {2025},
}

@article{Habib2020,
  author   = {Amiri, Habib and Gl{\"o}ckner, Helge and Schmeding, Alexander},
  doi      = {10.5817/AM2020-5-307},
  fjournal = {Archivum Mathematicum},
  issn     = {0044-8753},
  journal  = {Arch. Math. (Brno)},
  number   = {5},
  pages    = {307--356},
  title    = {Lie groupoids of mappings taking values in a {Lie} groupoid.},
  volume   = {56},
  year     = {2020},
}

@article{Hawkins2008,
  author     = {Hawkins, Eli},
  doi        = {10.4310/jsg.2008.v6.n1.a4},
  fjournal   = {The Journal of Symplectic Geometry},
  issn       = {1527-5256,1540-2347},
  journal    = {J. Symplectic Geom.},
  mrclass    = {46L65 (22A22 53D17 53D50)},
  mrnumber   = {2417440},
  mrreviewer = {Xiang\ Tang},
  number     = {1},
  pages      = {61--125},
  title      = {A groupoid approach to quantization},
  url        = {https://doi.org/10.4310/jsg.2008.v6.n1.a4},
  volume     = {6},
  year       = {2008},
}

@article{JoOr2014,
  author     = {Jotz Lean, M. and Ortiz, C.},
  doi        = {10.1016/j.indag.2014.07.009},
  fjournal   = {Koninklijke Nederlandse Akademie van Wetenschappen.
                Indagationes Mathematicae. New Series},
  issn       = {0019-3577,1872-6100},
  journal    = {Indag. Math. (N.S.)},
  mrclass    = {58H05 (53C12)},
  mrnumber   = {3264786},
  mrreviewer = {David\ Iglesias Ponte},
  number     = {5},
  pages      = {1019--1053},
  title      = {Foliated groupoids and infinitesimal ideal systems},
  url        = {https://doi.org/10.1016/j.indag.2014.07.009},
  volume     = {25},
  year       = {2014},
}

@article{Jotz2012,
  author   = {Jotz, M.},
  doi      = {10.3934/jgm.2012.4.313},
  fjournal = {Journal of Geometric Mechanics},
  issn     = {1941-4889},
  journal  = {J. Geom. Mech.},
  number   = {3},
  pages    = {313--332},
  title    = {The leaf space of a multiplicative foliation},
  volume   = {4},
  year     = {2012},
}

@phdthesis{KTPhD,
  author = {Kraasch-Tarnowsky, Annika},
  school = {Rheinischen Friedrich-Wilhelms-Universit¨at Bonn},
  title  = {The differentiable stack cohomology
            associated to a regular and proper Lie
            groupoid},
  year   = {2025},
}

@article{Laurent-Gengoux2009,
  author   = {Laurent-Gengoux, Camille and Sti{\'e}non, Mathieu and Xu, Ping},
  doi      = {10.1016/j.aim.2008.10.018},
  fjournal = {Advances in Mathematics},
  issn     = {0001-8708},
  journal  = {Adv. Math.},
  number   = {5},
  pages    = {1357--1427},
  title    = {Non-abelian differentiable gerbes},
  volume   = {220},
  year     = {2009},
}

@article{LiBland2014,
  author   = {Li-Bland, David and Meinrenken, Eckhard},
  doi      = {10.4310/AJM.2014.v18.n5.a2},
  fjournal = {The Asian Journal of Mathematics},
  issn     = {1093-6106},
  journal  = {Asian J. Math.},
  number   = {5},
  pages    = {779--816},
  title    = {Dirac {Lie} groups},
  volume   = {18},
  year     = {2014},
}

@article{Mackenzie1992,
  author     = {Mackenzie, Kirill C. H.},
  doi        = {10.1016/0001-8708(92)90036-K},
  fjournal   = {Advances in Mathematics},
  issn       = {0001-8708,1090-2082},
  journal    = {Adv. Math.},
  mrclass    = {58H05 (17B99 18D05 53C05 58F05)},
  mrnumber   = {1174393},
  mrreviewer = {Johannes\ Huebschmann},
  number     = {2},
  pages      = {180--239},
  title      = {Double {L}ie algebroids and second-order geometry. {I}},
  url        = {https://doi.org/10.1016/0001-8708(92)90036-K},
  volume     = {94},
  year       = {1992},
}

@book{Mackenzie2005,
  author    = {Mackenzie, Kirill C. H.},
  fseries   = {London Mathematical Society Lecture Note Series},
  isbn      = {0-521-49928-3},
  issn      = {0076-0552},
  publisher = {Cambridge: Cambridge University Press},
  series    = {Lond. Math. Soc. Lect. Note Ser.},
  title     = {The general theory of {Lie} groupoids and {Lie} algebroids},
  volume    = {213},
  year      = {2005},
}

@book{Munkres2000,
  author    = {Munkres, James R.},
  edition   = {2nd ed.},
  isbn      = {0-13-181629-2},
  publisher = {Upper Saddle River, NJ: Prentice Hall},
  title     = {Topology.},
  year      = {2000},
}

@article{Pugliese2023,
  author   = {Pugliese, Fabrizio and Sparano, Giovanni and Vitagliano, Luca},
  doi      = {10.1142/S0219199721500929},
  fjournal = {Communications in Contemporary Mathematics},
  issn     = {0219-1997},
  journal  = {Commun. Contemp. Math.},
  number   = {1},
  pages    = {36},
  title    = {Multiplicative connections and their {Lie} theory},
  volume   = {25},
  year     = {2023},
}

@book{Tu2017,
  author    = {Tu, Loring W.},
  doi       = {10.1007/978-3-319-55084-8},
  fseries   = {Graduate Texts in Mathematics},
  isbn      = {978-3-319-55082-4; 978-3-319-85562-2; 978-3-319-55084-8},
  issn      = {0072-5285},
  publisher = {Cham: Springer},
  series    = {Grad. Texts Math.},
  title     = {Differential geometry. {Connections}, curvature, and characteristic classes},
  volume    = {275},
  year      = {2017},
}

@article{Xiang2006,
  author     = {Tang, Xiang},
  doi        = {10.1007/s00039-006-0567-6},
  fjournal   = {Geometric and Functional Analysis},
  issn       = {1016-443X,1420-8970},
  journal    = {Geom. Funct. Anal.},
  mrclass    = {53D55 (17B63 53D17 58H05)},
  mrnumber   = {2238946},
  mrreviewer = {Stefan\ Waldmann},
  number     = {3},
  pages      = {731--766},
  title      = {Deformation quantization of pseudo-symplectic ({P}oisson)
                groupoids},
  url        = {https://doi.org/10.1007/s00039-006-0567-6},
  volume     = {16},
  year       = {2006},
}
\end{document}